\documentclass[reqno]{amsart}

\usepackage{amssymb}
\usepackage[psamsfonts]{eucal}
\usepackage{fancyhdr}

\newtheorem{teo}{Theorem}

\newtheorem{lemma}{Lemma}
\theoremstyle{definition}
\newtheorem{defi}{Definition}
\theoremstyle{remark}
\newtheorem{rem}{Remark}

\newcommand{\adj}{\operatorname{Adj}}

\title[Matrix orthogonality on the real line]{Vector interpretation of the matrix orthogonality on the real line}

\author[A. Branquinho]{\sc A. Branquinho}
\address[A. Branquinho]{CMUC and Departamento de Matem\'atica, Universidade de Coimbra,
Apartado 3008, EC Universidade, 3001-454 Coimbra, Portugal.}
\email[A. Branquinho]{ajplb@mat.uc.pt}

\author[F. Marcell\'{a}n]{\sc F. Marcell\'{a}n}
\address[F. Marcell\'{a}n]{Departamento de Matem\'aticas,
Escuela Polit\'ecnica Superior, Universidad Carlos III de Madrid, Avenida de
la Universidad, 30, 28911 Legan\'es, Spain.}
\email[F. Marcell\'{a}n]{pacomarc@ing.uc3m.es}

 \author[A. Mendes]{\sc A. Mendes}
 \address[A. Mendes]{Departamento de Matem\'atica, Escuela de Tecnolog\'ia y Gesti\'on,
 Instituto Polit\'ecnico de Leiria,  2411 - 901  Leiria - Portugal.}
 \email[A. Mendes]{aimendes@estg.ipleiria.pt}

\thanks{{\hspace{-.45 cm}}{\it 2000 Mathemathics Subject Classification}. 33C45, 39B42.  }

\begin{document}

\begin{abstract}
In this paper we study sequences of vector orthogonal po\-ly\-no\-mials. The vector orthogonality presented here provides a
reinterpretation of what is known in the literature as matrix orthogonality. These systems of orthogonal polynomials satisfy three-term recurrence relations with matrix coefficients that do not obey to any type of symmetry.  In this sense the vectorial reinterpretation allows us to study a non-symmetric case of the matrix orthogonality. We also prove that our systems of polynomials are indeed orthonormal with respect to a complex measure of orthogonality. Approximation problems of Hermite-Pad\'{e} type are also discussed. Finally, a Markov's type theorem is presented.
\end{abstract}

 \maketitle

\section{Introduction}

$\phantom{ola}$In the late eighties of the last century, the following problem attracted the interest of many researchers.

$\phantom{ola}$\emph{When a sequence of monic polynomials, $\{ p_n \}_{n \in \mathbb{N}}$, satisfying a recurrence~relation
 \begin{eqnarray}
\label{rel2n+1} x^N p_n(x) = c_{n,0} p_n(x) + \sum_{k=1}^N \left[
\overline{c}_{n,k}p_{n-k}(x)+ c_{n+k,k} p_{n+k} (x) \right],
 \end{eqnarray}
  where  $c_{n,0}\, (n=0,1,\ldots)$ is a real sequence and $c_{n,k}$, $(n=1,2,\ldots)$ are sequences of complex numbers for $k=1,2, \ldots ,N$ with
$c_{n,N}\neq 0,$ is related with some kind of orthogonality?}

$\phantom{ola}$Several authors (A. J. Dur\'{a}n, F. Marcell\'{a}n, W. Van Assche, and S. M. Zagorodnyuk, among others) were interested on this subject. Their contributions revealed an enormous interdisciplinarity between different kinds of orthogonality (like Sobolev orthogonality, orthogonality on rays of the complex plane) and several applications, mainly quadrature formulas. From the extensive bibliography on the subject we stand out the references~\cite{Cberg,Dur93,Dur95,DurWal95,EvansETal,MarcellanZagorodnyuk,nevai,sinapassche,Zagorodnyuk2,Zagorodnyuk3,Zagorodnyuk4,Zagorodnyuk5}.

$\phantom{ola}$In his work~\cite{Dur93}, A. J. Dur\'{a}n  presents for the first time a Favard's theorem
for sequences of polynomials $\{p_n\}_{n \in {\mathbb{N}}}$ satisfying recurrence relations like~\eqref{rel2n+1}. Few years later this result was reformulated by the author together with W. Van Assche in~\cite{DurWal95} where they stated the connection between sequences of matrix orthogonal polynomials and sequences of polynomials that satisfy a higher order recurrence relation. As an application,  the authors gave an interpretation of a Sobolev discrete inner product.

$\phantom{ola} $In this context the authors considered a positive integer number $N$ and the operators $R_{N,m}$, $m=0,1,\ldots,N-1$, defined on the linear space of polynomials,~${\mathbb{P}}$,~by
$$R_{N,m}(p)(x)=\sum_{n=0}^\infty \frac{p^{(nN+m)}(0)}{(nN+m)!}x^n, $$ i.e, the operator $R_{N,m}$ takes from $p$ just those powers with remainder $m$ (modulus~$N$) and then removes $x^m$ and changes $x^N$ to $x$ proving then, the following result.
 \begin{teo}[\cite{Dur93}]\label{favardrr2n+1}
 Suppose that $\{p_n\}_{n \in {\mathbb{N}}}$, with $\deg p_n = n$, is a sequence of polynomials satisfying a $(2N+1)-$term recurrence relation as
\eqref{rel2n+1} and let $\{P_n\}_{n \in {\mathbb{N}}}$ be a matrix polynomial sequence defined by
\begin{eqnarray*}
 P_n(x) = \left[%
\begin{matrix}
  R_{N,0}(p_{nN})(x) 
  &\cdots &  R_{N,N-1}(p_{nN})(x) \\
  \vdots 
  & \ddots & \vdots \\
  R_{N,0}(p_{(n+1)N-1})(x) 
  & \cdots & R_{N,N-1}(p_{(n+1)N-1})(x)
\end{matrix}%
\right] .
 \end{eqnarray*}
 Then, this sequence is orthonormal on the real line with respect to a positive definite matrix of measures and satisfies a three-term recurrence relation with matrix coefficients. \\
 Conversely, suppose that $\{P_n\}_{n \in {\mathbb{N}}},$ with $P_n=(P_{n}^{m,j})_{m,j=0}^{N-1},$ is a sequence of orthonormal matrix polynomials or, equivalently,  they satisfy a symmetric three-term recurrence relation with matrix coefficients. Then the scalar polynomials defined by
 \begin{eqnarray*}
\label{mix} p_{nN+m}(x)=
\sum_{j=0}^{N-1} x^j P_{n}^{m,j}(x^N) \quad (n \in {\mathbb{N}}, \,
0 \leq m \leq N-1)
 \end{eqnarray*}
 satisfy a $(2N+1)-$term recurrence relation of the form~\eqref{rel2n+1}.
 \end{teo}
$\phantom{ola}$Taking into account the current relevance of the subject our work is concerned with the analysis of higher order recurrence relations, in this case, of order $2N+1$
 \begin{eqnarray}
\label{rr2n+1cap1} h(x)p_n(x) =
c^{n+N-1}_{n+N} p_{n+N}(x) + \sum_{k=0}^{2N-1}
c^{n+N-1}_{n+N-1-k}p_{n+N-1-k}(x)
 \end{eqnarray}
 where $h$ is a polynomial of fixed degree $N$ and where $c^{n+N-1}_{j}$, $ n\geq 0,$ are real sequences for  $j=n-N,\ldots,n+N-1$ with $c^{n+N-1}_{n-N}\neq
0$ and  initial conditions on~$p_i$ for $i=0,\ldots, N-1$ are given.

$\phantom{ola}$We begin by pointing out that in the structure of the recurrence relation~\eqref{rr2n+1cap1} the polynomial $h$ is a generic polynomial with fixed degree $N$ and their coefficients do not satisfy any kind of symmetry.

$\phantom{ola}$Our aim is to analyze this more general case by studying the sequences of polynomials satisfying such a kind of recurrence relations in order to find out what type of orthogonality is associated with them. On the other hand, as an application, we expect to obtain some known results.

$\phantom{ola}$Let us consider the family of vector polynomials
$\displaystyle \mathbb{P}^{N}= \{ \left[
p_{1} \, \cdots \,
p_{N}
\right] ^{T} : p_{j} \in \mathbb{P} \} \, $,
and $ \mathcal{M}_{N \times N}({\mathbb{R}})$ the
set of $ N \times N$ matrices with real entries.
Given a polynomial  $h,$ with  $\deg h = N,$ we can split the linear
space of polynomials, ${\mathbb{P}}$, using the basis
 \begin{eqnarray}\label{base}
\{1,x,\ldots,x^{N-1}, h(x), x h(x), \ldots, x^{N-1}h(x),h^2(x), xh^2(x), \ldots\}.
\end{eqnarray}
Then, let $ \{ \mathcal{P}_j\}_{j\in {\mathbb{N}}}$ be a sequence of vector polynomials such that
$\displaystyle \mathcal{P}_{j}(x)= (h(x))^j$ $\displaystyle \mathcal{P}_{0}(x) \, $,
where $\mathcal{P}_{0}(x)=\left[
 1 \, x\, \cdots \,
x^{N-1} \right] ^{T} \, , \ j \in \mathbb{N}.$
Let $ \{p_m\}_{m \in {\mathbb{N}}}$ be a sequence of polynomials, $ \deg
p_{m}=m$, $m\in \mathbb{N}$. We define the \textit{associated vector polynomial
sequence}
$ \{
\mathcal{B}_{m} \}_{m \in {\mathbb{N}}} $ by
 \begin{equation*}
\mathcal{B}_{m}= \left[
 p_{mN} \,  \cdots \,
p_{(  m+1 )N-1}
\right] ^{T},\,\ n\in \mathbb{N} \, .
\end{equation*}
$\phantom{ola}$A scalar polynomial $p_{mN+k}$ of degree $mN+k$, with $0\leq k \leq N-1,$ can be expanded in the basis~\eqref{base} as follows
$$p_{mN+k}(x)=\sum_{i=0}^m \sum_{j=0}^{N-1} a_{i,j} x^j h^i(x).$$
If we consider the operator $R_{h,N,j}$ that takes from  $p_{mN+k}$ the terms of the form $a_{i,j}x^j h^i(x)$ and then removes the common factor $x^j$ and change $h(x)$ to $x$, we get
$$p_{mN+k}(x)=\sum_{j=0}^{N-1}x^j R_{h,N,j}(p_{mN+k}) (h(x)).$$
$\phantom{ola}$It is easy to see that we can write $\mathcal{B}_{m}$ in the matrix form
 \begin{equation}\label{vv}
\mathcal{B}_{m}(x)=V_{m}(h(x))\mathcal{P}_{0}(x),
\end{equation}
where $ V_{m}$ is a~$N\times N$ matrix polynomial of degree $m$ given~by
 \begin{equation*}
V_{m}(h(x))=\left[
 \begin{matrix}
  R_{h,N,0}(p_{nN})(h(x))&\cdots &  R_{h,N,N-1}(p_{nN})(h(x)) \\
  \vdots  & \ddots & \vdots \\
  R_{h,N,0}(p_{(n+1)N-1})(h(x))&  \cdots & R_{h,N,N-1}(p_{(n+1)N-1})(h(x)) \\
 \end{matrix}
 \right]
 \end{equation*}
and ${\mathcal{P}_{0}}(x)=\left[1 \, x\, \cdots \, x^{N-1} \right] ^{T}$.
Equivalently, we can write the elements of the sequence of matrix polynomials $\{V_m\}_{m \in {\mathbb{N}}}$ in
the form
 \begin{eqnarray*}
V_m(h(x))=\sum_{j=0}^m B_j^m (h(x))^j,
 \end{eqnarray*}
where $(B_j^m)$ is a family of  matrices with real entries.

$\phantom{ola}$First we want to prove that if a sequence of scalar polynomials $\{p_n\}_{n \in {\mathbb{N}}}$ satisfies a recurrence relation like~\eqref{rr2n+1cap1} then there exists a sequence of vector polynomials denoted by
$\{{\mathcal{B}}_m\}_{m \in {\mathbb{N}}}$ and a sequence of matrix polynomials $\{V_m\}_{m \in {\mathbb{N}}}$ defined by~\eqref{vv} that satisfies a recurrence relation with matrix coefficients and the converse is also true.

$\phantom{ola}$Notice that we can rewrite~\eqref{rr2n+1cap1} changing $n$ by $n+N-1$,
 \begin{eqnarray}
\label{relrec_seg} h(x)p_{n+N-1}(x) = c^{n+2(N-1)}_{n+2N-1}
p_{n+2N-1}(x) + \sum_{k=0}^{2N-1}
c^{n+2(N-1)}_{n+2(N-1)-k}p_{n+2(N-1)-k}(x)
 \end{eqnarray}
and then, consider the $N$ equations associated with ~(\ref{rr2n+1cap1}) and ~(\ref{relrec_seg}).

$\phantom{ola}$A straightforward calculation yields that the above system of $N$ linear equations can be written in the matrix form
 \begin{multline*} 
h(x) \left[ \begin{matrix}
   p_n (x)\\
      \vdots \\
      p_{n+N-1}(x)
    \end{matrix}\right]= \left[%
\begin{matrix}
  c^{n+N-1}_{n+N} & \ldots& 0\\
  \vdots & \ddots& \vdots \\
  c^{n+2N-2}_{n+N} & \ldots & c^{n+2N-2}_{n+2N-1}
\end{matrix}%
\right] \left[ \begin{matrix}
   p_{n+N}(x) \\
      \vdots \\
      p_{n+2N-1}(x)
    \end{matrix}\right] \\ + \left[%
\begin{matrix}
  c^{n+N-1}_{n}  & \ldots& c^{n+N-1}_{n+N-1}\\
  \vdots & \ddots & \vdots \\
  c^{n+2N-2}_n  & \ldots & c^{n+2N-2}_{n+N-1}
\end{matrix}%
\right] \left[\begin{matrix}
   p_{n}(x) \\
      \vdots \\
      p_{n+N-1}(x)
    \end{matrix}\right] \\ + \left[%
\begin{matrix}
  c^{n+N-1}_{n-N}  & \ldots& c^{n+N-1}_{n-1}\\
  \vdots  & \ddots & \vdots \\
  0  & \ldots & c^{n+2N-2}_{n-1}
\end{matrix}%
\right] \left[ \begin{matrix}
   p_{n-N}(x) \\
      \vdots \\
      p_{n-1}(x)
    \end{matrix}\right]\,.
    \end{multline*}
$\phantom{ola}$Introducing the change of index $n=mN$ in the above relation we get
 \begin{eqnarray}\label{rrbv}
 h(x)
{\mathcal{B}}_m(x) = A_{m} {\mathcal{B}}_{m+1}(x) + B_m
{\mathcal{B}}_m(x) + C_m {\mathcal{B}}_{m-1}(x), \, \, m \geq 1,
 \end{eqnarray}
$\phantom{ola}$Similarly, if we take into consideration Theorem~\ref{favardrr2n+1} given by A. J. Dur\'{a}n, and instead of using the canonical basis for the linear space of polynomials, ${\mathbb{P}},$ we deal with the basis~\eqref{base}, then we get a sequence of polynomials
$\{V_m\}_{m \in {\mathbb{N}}}$ that satisfies the three-term recurrence relation
 \begin{eqnarray*}
V_m(z) = A_{m} V_{m+1}(z) + B_m V_m(z) + C_m V_{m-1}(z), \, \, m \geq 1
 \end{eqnarray*}
with some given initial conditions.

$\phantom{ola}$Notice  that if we multiply this last relation by ${\mathcal{P}}_0$ then we obtain the recurrence relation for the sequence of vector polynomials $\{\mathcal{B}_m\}_{m \in {\mathbb{N}}}$ given by~\eqref{rrbv}.
Finally, from~\eqref{rrbv}  and  taking into account the structure of the recurrence relation as well as the expression of the vector ${\mathcal{B}}_m$, we get the $(2N+1)-$term recurrence relation~\eqref{rr2n+1cap1}.
 \begin{teo}\label{teorema2}
Let $\{p_n\}_{n \in {\mathbb{N}}}$ be a sequence of scalar polynomials,
$\{{\mathcal{B}}_m\}_{m \in {\mathbb{N}}}$ the sequence of vector polynomials with
$${\mathcal{B}}_m(x) = \left[p_{mN}(x) \, p_{mN+1}(x) \, \ldots \, p_{(m+1)N-1}(x)\right]^T,$$ and $\{V_m\}_{m \in {\mathbb{N}}}$ the sequence of matrix polynomials given in ~\eqref{vv}.
Then, the following statements are equivalent:
 \begin{itemize}
\item[(a)] The sequence of scalar polynomials $\{p_n\}_{n \in {\mathbb{N}}}$ satisfies~\eqref{rr2n+1cap1}.
\item[(b)] The sequence of vector polynomials $\{{\mathcal{B}}_m\}_{m \in {\mathbb{N}}}$ satisfies
 \begin{eqnarray*} h(x)
{\mathcal{B}}_m(x) = A_{m} {\mathcal{B}}_{m+1}(x) + B_m
{\mathcal{B}}_m(x) + C_m {\mathcal{B}}_{m-1}(x), \, \, m \geq 1,
\end{eqnarray*}
with initial conditions
 $\displaystyle {\mathcal{B}}_{-1}(x)=0_{N
\times 1}$ and $\displaystyle {\mathcal{B}}_{0}(x)$ 
given.
\item[(c)] The sequence of matrix polynomials $\{V_m\}_{m \in {\mathbb{N}}}$ satisfies
 \begin{eqnarray*} z
V_m(z) = A_{m} V_{m+1}(z) + B_m V_m(z) + C_m V_{m-1}(z), \, \, m \geq 1,
\end{eqnarray*}
with initial conditions
 $\displaystyle V_{-1}(z)=0_{N
\times N} \ \mbox{and } \ V_{0}(z) $ a fixed matrix.
\end{itemize}
The  matrices $A_m$,
$B_m,$ and $C_m$ in the recurrence relations are given, respectively,~by
 \begin{multline*}  
  \left[%
 \begin{matrix}
  c^{(m+1)N-1}_{(m+1)N} & \cdots & 0\\
  \vdots & \ddots & \vdots \\
  c^{(m+2)N-2}_{(m+1)N} &  \cdots & c^{(m+2)N-2}_{(m+2)N-1}
\end{matrix}%
\right] \, , \  
 \left[%
\begin{matrix}
  c^{(m+1)N-1}_{mN} &  \cdots& c^{(m+1)N-1}_{(m+1)N-1}\\
  \vdots  & \ddots & \vdots \\
  c^{(m+2)N-2}_{mN}  & \cdots & c^{(m+2)N-2}_{(m+1)N-1}
\end{matrix}%
\right], \\  \mbox{ and} \quad  
 \left[%
\begin{matrix}
  c^{(m+1)N-1}_{(m-1)N} & \cdots& c^{(m+1)N-1}_{mN-1}\\
  \vdots  & \ddots & \vdots \\
  0 &  \cdots & c^{(m+2)N-2}_{mN-1}
\end{matrix}%
\right].
 \end{multline*}
 \end{teo}
$\phantom{ola}$Now we consider the sequence of matrix polynomials $\{V_m\}_{m \in {\mathbb{N}}}$ defined by
\begin{eqnarray}
 \label{rrvm} xV_m(x)=A_m V_{m+1}(x) +B_m
V_m(x)+C_{m} V_{m-1}(x)\quad m\geq 0,
\end{eqnarray} with initial conditions
$\displaystyle V_{-1}(x)=0_{N
\times N} \quad \mbox{and} \quad V_{0}(x)=I_{N \times N}\,$.

$\phantom{ola}$The first question is to know when a sequence of matrix polynomials defined by~\eqref{rrvm} is related to the matrix orthogonality.

$\phantom{ola}$If $C_m=A_{m-1}^T$ and $B_m$ is  a positive definite  matrix of measures $\widetilde{W}$ supported on the real line then the polynomials $\{V_m\}_{m \in {\mathbb{N}}}$
are orthonormal with respect to a left inner product, i.e.,
 \begin{eqnarray}\label{pim}
\langle V_i, V_j \rangle = \int_{{\mathbb{R}}} V_i(x)d\widetilde{W}(x)V_j^T (x)=\delta_{i,j}I_{N \times N}.
\end{eqnarray}
$\phantom{ola}$In the last years several authors have studied analytic properties of matrix orthonormal polynomials (see for example~\cite{Dur93,Dur95,Dur96,DurDaneriVias}) and their connections with the spectral theory of linear differential operators with matrix polynomials as coefficients.

$\phantom{ola}$In the case when neither $C_m=A_{m-1}^T$ nor $B_m$ are symmetric we cannot guarantee that the system of matrix polynomials $\{V_m\}_{m \in {\mathbb{N}}}$ satisfying the recurrence relation~\eqref{rrvm} is orthogonal with respect to a inner product induced by a  positive definite matrix of measures $\widetilde{W}$.

$\phantom{ola}$In~\cite{Dettealunos} the authors presented a result that characterizes the existence of a matrix of measures $\widetilde{W}$ such that the system of polynomials $\{V_m\}_{m \in {\mathbb{N}}}$ is orthogonal in the sense of~\eqref{pim}. In fact, if the matrices $A_m$ and $C_m$, for $m \in {\mathbb{N}}$, in the recurrence relation~\eqref{rrvm}, are non-singular then there exists a matrix of measures on the real line with a positive definite Hankel matrix as moment matrix such that the system of polynomials $\{V_m\}_{m \in {\mathbb{N}}}$ defined by~\eqref{rrvm} is orthogonal with respect to the measure $\widetilde{W}$ in the sense of~\eqref{pim} if and only if there exists a sequence of non-singular matrices $\{R_m\}_{m \in {\mathbb{N}}}$ such that the following relations hold: \\
$\phantom{ola}$$\bullet$ $\displaystyle R_mB_mR_m^{-1}\quad \mbox{is symmetric,}\quad \forall\, m \in {\mathbb{N}}_0 \, $, \\
$\phantom{ola}$$\bullet$ $\displaystyle R_m^TR_m=C_m^{-T} \cdots C_1^{-T}(R_0^TR_0)A_0 \cdots A_{m-1}, \,\,\forall\,m \in {\mathbb{N}}_0 \, $.

$\phantom{ola}$In this contribution we prove that a recurrence relation ~\eqref{rrvm} characterizes a different kind of orthogonality. The structure of the paper is as follows:
In section~$2$, we present the algebraic theory of the sequences of vector polynomials. In this context, we define a vector linear functional and we introduce the concept of right and left-orthogonality with respect to this linear functional.
In section~$3$, we present a reinterpretation of the matrix or\-tho\-go\-nality in terms of the vector orthogonality showing that there are two sequences of matrix orthogonal polynomials with respect to a matrix of measures, not necessarily positive definite,  which are bi-orthogonal with respect to a vector linear functional.
In section~$4$, we analyze two type Hermite-Pad\'{e} approximation problems and, finally, a Markov's type Theorem is deduced.

\section{Vector orthogonality}

$\phantom{ola}$Let $(\mathbb{P}^{N})^*$ be the linear space of vector linear functionals defined on the linear space of vector polynomials with complex coefficients $\mathbb{P}^{N}$, i.e., $(\mathbb{P}^{N})^*$ is the {\it dual space} of $\mathbb{P}^{N}$. In this space we define a
\textit{vector of functionals} as follows.
 \begin{defi}
\label{vec_fun_lin}Let $ u^{j}:\mathbb{P}\rightarrow
\mathbb{R}$ with $ j=1,\ldots ,N$ be linear functionals.
We define the \textit{vector of functionals} $
{\mathcal{U}}= \left[
 u^{1} \cdots \,
u^{N} \right] ^{T}$ in $ \mathbb{P}^{N}$
with values in ~$ \mathcal{M} _{N\times N}(\mathbb{R})$,~by
 \begin{equation*}
\mathcal{U}(\mathcal{P}):= (\mathcal{U}.\mathcal{P} ^{T} )
^{T}=\left[
 \begin{matrix}
\langle u^1, p_1 \rangle  & \cdots & \langle u^N, p_1 \rangle  \\
\vdots & \ddots & \vdots \\
\langle u^1, p_N \rangle  & \cdots & \langle u^N, p_N \rangle
\end{matrix}
\right] \, ,
 \end{equation*}
where \ ``$\displaystyle( . )$'' \ means the symbolic product of
$ \mathcal{U}$ and $
\mathcal{P}^{T}.$
\end{defi}
$\phantom{ola}$Let $\displaystyle \widehat{A}(x) = \sum_{k=0}^{l}A_{k}\, x^{k}$, where
$ A_{k}\in \mathcal{M}_{N\times N}(\mathbb{R}),$ be a matrix polynomial and
$ \mathcal{U }$  be a vector of li\-ne\-ar functionals.
Let us consider the vector of li\-ne\-ar functionals, the so called \textit{left multiplication of $\mathcal{U}$
by $\widehat{A}$}, that we will denote  by $\widehat{A}\, \mathcal{U}$,
such that
 \begin{equation*}
 (\widehat{A} \, \mathcal{U} )(\mathcal{P} ):= (
  \widehat{A}\,\mathcal{U} . \mathcal{P}
^{T} )^{T}=\sum_{k=0}^{l} (x^{k}\,\mathcal{U} )( \mathcal{P}) \,
(A_{k} )^{T} \, .
 \end{equation*}
$\phantom{ola}$We will introduce the concept of sequence of vector polynomials left-orthogonal with respect to the vector of linear functionals ${\mathcal{U}}$ and we will prove that ${\mathcal{U}}$ is {\it quasi-definite}, i.e, there exists a unique sequence of vector polynomials, up to the multiplication on the left by a non-singular matrix, that is left-orthogonal with respect to ${\mathcal{U}}$.
 \begin{defi} Let $\{ p_n \}_{n \in \mathbb{N}}$ be a sequence of scalar polynomials with $\deg p_n = n \, $, $n \in \mathbb{N}$.
 Let $h$ be a polynomial of fixed degree $N$, $\{{\mathcal{B}}_m\}_{m \in {\mathbb{N}}}$ be a sequence of vector polynomials with
$\displaystyle {\mathcal{B}}_m (x)= \left[ p_{mN}(x)\, p_{mN+1}(x)\, \cdots \,
p_{(m+1)N-1}(x) \right]^T ,$ and let ${\mathcal{U}}=\left[u^{1} \cdots \,
u^{N} \right] ^{T}$ be a vector of linear functionals.  $\{{\mathcal{B}}_m\}_{m \in {\mathbb{N}}}$ is said to be {\it left-orthogonal} with respect to the vector of linear functionals ${\mathcal{U}}$ if
 \begin{itemize}
\item[(a)] $(h^{k} {\mathcal{U}}) \left( {\mathcal{B}}_m \right)= 0_{N \times
N}$, $k=0,1, \ldots ,m-1$.
\item[(b)] $(h^{m} {\mathcal{U}}) \left( {\mathcal{B}}_m \right)=
\Delta_m$, $m \in \mathbb{N},$ where $\Delta_m$ is a non-singular upper triangular matrix.
\end{itemize}
\end{defi}

$\phantom{ola}$We introduce the notion of {\it moment} associated with the vector of linear functionals~${\mathcal{U}}$.
Taking into account that $\{{{\mathcal{P}}_j}\}_{j \in {\mathbb{N}}},$ with
 $\displaystyle {\mathcal{P}}_j(x)=(h(x))^j{\mathcal{P}}_0(x)$
and ${\mathcal{P}}_0(x)=[1 \, x\, \cdots \, x^{N-1}]^T,$  is a basis in
the linear space of vector polynomials ${\mathbb{P}}^N,$ we denote $(x^k{\mathcal{U}})({\mathcal{P}}_j) = {\mathcal{U}}^k_j$
 the {\it j-$th$ moment} associated with the vector of linear functionals~$x^k {\mathcal{U}}$.

$\phantom{ola}$The {\it Hankel matrices} associated with ${\mathcal{U}}$ are the matrices
\begin{eqnarray*} 
 D_m =\left[%
\begin{array}{ccc}
  {\mathcal{U}}_0 & \cdots &  {\mathcal{U}}_m  \\
  \vdots & \ddots & \vdots \\
   {\mathcal{U}}_m & \cdots &  {\mathcal{U}}_{2m}  \\
\end{array}%
\right],\,  m \in {\mathbb{N}},
\end{eqnarray*}
where ${\mathcal{U}}_j$ are  $j$-th moments associated with the vector of linear functionals ${\mathcal{U}}$.
${\mathcal{U}}$ is said to be {\it quasi-definite} if all leading principal submatrices of $D_m, \, m \in {\mathbb{N}},$ are non-singular.

$\phantom{ola}$The following result provides a necessary and sufficient condition for the existence of a sequence of vector polynomials which are left-orthogonal with respect to the vector of linear functionals $\mathcal{U}$.
\begin{teo}
 Let $\mathcal{U}$ be a vector of linear functionals. Then ${\mathcal{U}}$ is quasi-definite if, and only if, there exists a unique sequence of vector polynomials $\{{\mathcal{B}}_m\}_{m \in {\mathbb{N}}}$ such that ${\mathcal{B}}_m =\sum_{j=0}^m \alpha^m_j {\mathcal{P}}_j$, where $\alpha^m_j \in {\mathcal{M}}_{N\times N}({\mathbb{R}})$ with $\alpha_m^m$ is non-singular lower triangular matrix and a unique sequence, $\{\Delta_m\}_{m \in {\mathbb{N}}}$, of non-singular upper triangular matrices such that
 $\displaystyle
 (h^{k} {\mathcal{U}}) \left( {\mathcal{B}}_m \right)= \Delta_m \delta_{k,m}, \, k=0,1,\ldots, m, \,\, m \in {\mathbb{N}} \, .$
\end{teo}
\begin{proof} To prove that ${\mathcal{U}}$ is quasi-definite. Let $\{{\mathcal{B}}_m\}_{m \in {\mathbb{N}}}$ be a sequence of vector polynomials with  $ {\mathcal{B}}_m =
\sum_{j=0}^m \alpha^m_j {\mathcal{P}}_j$, where $\alpha^m_j \in
{\mathcal{M}}_{N\times N}({\mathbb{R}})$ and
$\{{\mathcal{P}}_j\}_{j \in {\mathbb{N}}}$ is a basis in~${\mathbb{P}}^N,$ such that
$\displaystyle {\mathcal{P}}_j(x)=(h(x))^j {\mathcal{P}}_0(x), \quad
{\mathcal{P}}_0(x)=\left[1 \, x\, \cdots\, x^{N-1} \right]^T \, .$

$\phantom{ola}$From the orthogonality conditions, the vector sequence of polynomials $\{{\mathcal{B}}_m\}_{m \in {\mathbb{N}}}$ is left-orthogonal with respect to the vector of linear functionals $\mathcal{U}$ if, for $k=0, \ldots, m-1$,
\begin{eqnarray*}  \left( h^{k} \mathcal{U}\right)
\left( {\mathcal{B}}_m \right)= \left( h^{k} \mathcal{U}\right)
(\sum^m_{j=0} \alpha^m_j {\mathcal{P}}_j )=\sum^m_{j=0}
\alpha^m_j ( h^{k} \mathcal{U}) \left( {\mathcal{P}}_j \right)= 0_{N
\times N},
\end{eqnarray*} and for all $m \in {\mathbb{N}}$,
\begin{eqnarray*} \left( h^{m} \mathcal{U}\right)
\left( {\mathcal{B}}_m \right)= \left( h^{m} \mathcal{U}\right)
(\sum^m_{j=0} \alpha^m_j {\mathcal{P}}_j )=\sum^m_{j=0}
\alpha^m_j \left( h^{m} \mathcal{U}\right) \left( {\mathcal{P}}_j
\right)= \Delta_m.
\end{eqnarray*}
Taking into account $\displaystyle (h^k {\mathcal{U}})({\mathcal{P}}_j)={\mathcal{U}}({\mathcal{P}}_{j+k}),$ the above conditions can be read as
\begin{eqnarray}\label{exite2}
 \left[
\begin{array}{cccc} \alpha_0^m & \alpha_1^m & \cdots & \alpha_m^m \\
\end{array} \right]\left[%
\begin{array}{cccc}
  {\mathcal{U}}({\mathcal{P}}_0)   & \cdots & {\mathcal{U}}({\mathcal{P}}_m)\\
  \vdots  &  \ddots & \vdots \\
  {\mathcal{U}}({\mathcal{P}}_m)   & \cdots & {\mathcal{U}}({\mathcal{P}}_{2m}) \\
\end{array}%
\right]=\left[
 \begin{array}{cccc}
         0 & 0 & \cdots & \Delta_m \\
 \end{array}\right].
\end{eqnarray}
 For $m=0,$ in~\eqref{exite2} we have $\alpha_0^0 {\mathcal{U}}_0=\Delta_0$. Using the non-singularity of the matrices $\alpha_{0}^{0}$ and $\Delta _{0},$ $ \mathcal{U}_{0}$ is a non-singular matrix. In an analog way, taking $m=1$ in~\eqref{exite2}, we have
\begin{equation*}
\left\{
\begin{array}{l}
\alpha_{0}^{1}\,\mathcal{U}_{0}+\alpha_{1}^{1}\,\mathcal{U}_{1}=0_{N\times N} \\
\alpha_{0}^{1}\,\mathcal{U}_{1}+\alpha_{1}^{1}\,\mathcal{U}_{2}=\Delta_{1},
\end{array} \quad \mbox{i.e.} \quad \alpha_1^1 (\mathcal{U}_{2}-\mathcal{U}_{1}\mathcal{U}_{0}^{-1}\mathcal{U}_{1})=\Delta_1.
\right.
\end{equation*}
Since $\Delta_1$ and $\alpha_1^1$ are non-singular matrices then
$\det(\mathcal{U}_{2}-\mathcal{U}_{1}\mathcal{U}_{0}^{-1}\mathcal{U}_{1}) \neq 0$ and, as a consequence, the second leading principal submatrix is  non-singular.
This argument can be inductively used and we obtain that $\mathcal{U}$ is quasi-definite.

$\phantom{ola}$Conversely, to find the vector sequence of polynomials such that $\{{\mathcal{B}}_m\}_{m \in {\mathbb{N}}}$ with ${\mathcal{B}}_m =\sum_{j=0}^m \alpha^m_j {\mathcal{P}}_j$, where $\alpha^m_j \in {\mathcal{M}}_{N\times N}({\mathbb{R}})$ and where $\alpha_m^m$ is non-singular lower triangular matrix such that
 $\displaystyle (h^{k} {\mathcal{U}}) \left( {\mathcal{B}}_m \right)= \Delta_m \delta_{k,m}, \, k=0,1,\ldots, m, \,\, m \in {\mathbb{N}} \, ,$
is equivalent to solve
\begin{eqnarray*}\left[   \begin{array}{cccc} \alpha_0^m & \alpha_1^m & \cdots & \alpha_m^m \\
\end{array} \right]\left[%
\begin{array}{cccc}
  {\mathcal{U}}({\mathcal{P}}_0)   & \cdots & {\mathcal{U}}({\mathcal{P}}_m)\\
  \vdots  &  \ddots & \vdots \\
  {\mathcal{U}}({\mathcal{P}}_m)   & \cdots & {\mathcal{U}}({\mathcal{P}}_{2m}) \\
\end{array}%
\right]=\left[\begin{array}{cccc}
         0 & 0 & \cdots & \Delta_m \\
               \end{array}\right].
\end{eqnarray*}
 For $m=0$, we have $\alpha_0^0 {\mathcal{U}}_0=\Delta_0$. Using the non-singularity of ${\mathcal{U}}_0$, and the decomposition $LU$, we can find uniquely $\alpha_0^0$ a non-singular lower triangular matrix and $\Delta_0$ a non-singular upper triangular matrix such that $\alpha_0^0 {\mathcal{U}}_0=\Delta_0$.

$\phantom{ola}$For $m=1$ we have
\begin{equation*}
\left\{
\begin{array}{l}
\alpha_{0}^{1}\,\mathcal{U}_{0}+\alpha_{1}^{1}\,\mathcal{U}_{1}=0_{N\times N} \\
\alpha_{0}^{1}\,\mathcal{U}_{1}+\alpha_{1}^{1}\,\mathcal{U}_{2}=\Delta_{1},
\end{array} \quad \mbox{i.e.} \quad \alpha_1^1 (\mathcal{U}_{2}-\mathcal{U}_{1}\mathcal{U}_{0}^{-1}\mathcal{U}_{1})=\Delta_1.
\right.
\end{equation*}
Again, using that the second leading principal submatrix $\mathcal{U}_{2}-\mathcal{U}_{1}\mathcal{U}_{0}^{-1}\mathcal{U}_{1}$ is non-singular and the $LU$ decomposition we can find uniquely $\alpha_1^1$ a non-singular lower triangular matrix and $\Delta_1$ a non-singular upper triangular matrix such that $\alpha_1^1 =(\mathcal{U}_{2}-\mathcal{U}_{1}\mathcal{U}_{0}^{-1}\mathcal{U}_{1})=\Delta_1$.
We also obtain from $\alpha_{0}^{1}\,\mathcal{U}_{0}+\alpha_{1}^{1}\,\mathcal{U}_{1}=0_{N\times N}$, uniquely the matrix $\alpha_{0}^{1}$.
This argument can be inductively used and we obtain the stated result.
\end{proof}
\begin{teo}\label{FAVARD1}
Let $\mathcal{U}$ be a quasi-definite vector of linear functionals and let $\{{\mathcal{B}}_m\}_{m \in {\mathbb{N}}}$ be a sequence of vector
polynomials. Then, the following statements are equivalent:
\begin{itemize}
\item[(a)] The vector sequence of polynomials $\{{\mathcal{B}}_m\}_{m \in {\mathbb{N}}}$ is left-orthogonal with respect to the vector of linear functionals $\mathcal{U}$, i.e.,
\begin{eqnarray}
 \label{cotra1}
(h^k {\mathcal{U}}) \left( {\mathcal{B}}_m
\right)=\Delta_m \delta_{k,m} \, , \quad k=0,1,\ldots,m, \, m \in
{\mathbb{N},}
\end{eqnarray}
with $\Delta_m$ a non-singular $N\times N$ upper triangular matrix given~by
$$\Delta_m =C_m \,\cdots \,C_1 \,\Delta_0,
\, \ m \geq 1, $$ where $\Delta_0$ is a~$N\times N$ non-singular matrix and $\{C_m\}_{m \in {\mathbb{N}}}$ is a sequence of non-singular upper triangular matrices.
\item[(b)]  There exist sequences of~$N\times N$ matrices, by $\{A_m\}_{m \in {\mathbb{N}}},$$\{B_m\}_{m \in {\mathbb{N}}}$, and $\{C_m\}_{m \in {\mathbb{N}}}$ with $C_m$ a non-singular upper triangular matrix, such that
 \begin{eqnarray}
\label{recorr1} h(x) {\mathcal{B}}_m(x) = A_m {\mathcal{B}}_{m+1}(x)
+ B_m {\mathcal{B}}_m(x) + C_m {\mathcal{B}}_{m-1}(x), \quad m\geq
1,
\end{eqnarray}
with
$\displaystyle {\mathcal{B}}_{-1}(x)=0_{N \times 1} \
\mbox{and } \ {\mathcal{B}}_{0}(x)={\mathcal{P}}_0(x)$,
 where
$\displaystyle {\mathcal{P}}_0(x)=\left[1\,x\, \cdots\, x^{N-1}\right]^T$.
\end{itemize}
\end{teo}
 \begin{proof}
To prove  $(a)\Rightarrow(b),$
first we  consider the polynomial $h {\mathcal{B}}_{m}$ and then we take into account that
$h{\mathcal{B}}_{m}$ is a polynomial of degree $m+1$ that can be written
 \begin{eqnarray}
 \label{soma1} h(x) {\mathcal{B}}_{m}(x)
= \sum^{m+1}_{k=0} A^m_k {\mathcal{B}}_{k}(x),\, A^m_k \in
{\mathcal{M}}_{N\times N}({\mathbb{R}}).
\end{eqnarray}
We will prove that $A_k^m = 0_{N \times N}$, for $k=0,1, \ldots ,m-2$. Indeed, if we
apply the vector of functionals $\mathcal{U}$ to both sides of ~\eqref{soma1} then we get
$A_0^m = 0_{N \times N}$.

$\phantom{ola}$Thus we can rewrite~(\ref{soma1})
 \begin{eqnarray}
 \label{soma2}
h(x) {\mathcal{B}}_{m}(x) = \sum^{m+1}_{k=1} A^m_k {\mathcal{B}}_{k}(x)\,
.\end{eqnarray}
Again, by applying the vector of functionals $\mathcal{U}$ to both sides of ~\eqref{soma2}
we obtain $$A_1^m = 0_{N \times N}.$$
Iterating this procedure, i.e., first by multiplying by $h^2$, afterwards by $h^3$ and then, successively, applying
$\mathcal{U}$ we get
 \begin{eqnarray*}
  A_k^m = 0_{N \times N}\, , \quad
\mbox{for} \quad k=0,1, \ldots , m-2 \, .
 \end{eqnarray*}
 Thus, we can rewrite~(\ref{soma1})
 \begin{eqnarray} \label{recorr2}h(x)
{\mathcal{B}}_{m}(x) = A_{m-1}^m{\mathcal{B}}_{m-1}(x) + A_m^m
{\mathcal{B}}_{m}(x) + A_{m+1}^m {\mathcal{B}}_{m+1}(x).
 \end{eqnarray}
If we multiply~(\ref{recorr2}) by $h^{m-1}$ and we apply the vector of linear functionals  $\mathcal{U}$ then we obtain
\begin{equation*}
 A_{m-1}^m
= (h^m {\mathcal{U}})\left(  {\mathcal{B}}_{m}\right) \left(
(h^{m-1} {\mathcal{U}})\left(   {\mathcal{B}}_{m-1}\right)
\right)^{-1}
= \Delta_m \Delta_{m-1}^{-1} \quad m\geq 1 \, .
\end{equation*}
Using the same technique we obtain
\begin{eqnarray*} A_m^m &=&
\left[(h^{m+1}{\mathcal{U}})\left({\mathcal{B}}_{m}\right)-
A_{m-1}^m(h^{m}{\mathcal{U}})\left({\mathcal{B}}_{m-1}\right)\right]
\left[ (h^{m}{\mathcal{U}})\left(
{\mathcal{B}}_{m}\right)\right]^{-1} \\ &=&
\left[(h^{m+1}{\mathcal{U}})\left({\mathcal{B}}_{m}\right)-
 \Delta_m \Delta_{m-1}^{-1}(h^{m} {\mathcal{U}})\left(
{\mathcal{B}}_{m-1}\right)\right] \Delta_{m}^{-1}
\end{eqnarray*}
and
 $\displaystyle A_{m+1}^m =
 \left[(h^{m+2}{\mathcal{U}})
 \left({\mathcal{B}}_{m}\right)-
A_{m-1}^m (h^{m+1}{\mathcal{U}})\left(
{\mathcal{B}}_{m-1}\right)- A_m^m (h^{m+1}{\mathcal{U}})\left(
{\mathcal{B}}_{m}\right)\right]\Delta_{m+1}^{-1} $.

 The comparison with the coefficients in~(\ref{recorr1}) yields the following explicit expressions for the coefficients in the recurrence relation:
\begin{eqnarray*} A_m &=&
\left[(h^{m+2}{\mathcal{U}})\left({\mathcal{B}}_{m}\right)-
A_{m-1}^m (h^{m+1}{\mathcal{U}})\left(
{\mathcal{B}}_{m-1}\right)- A_m^m (h^{m+1}{\mathcal{U}})\left(
{\mathcal{B}}_{m}\right)\right] \Delta_{m+1}^{-1}, \\ B_m &=&
\left[(h^{m+1}{\mathcal{U}})\left({\mathcal{B}}_{m}\right)-
A_{m-1}^m (h^{m}{\mathcal{U}})\left({\mathcal{B}}_{m-1}\right)\right]\Delta_{m}^{-1},
 \\ C_m &=& \Delta_m
\Delta_{m-1}^{-1}.
\end{eqnarray*}
To prove $(b)\Rightarrow(a),$ we must start by constructing a vector of linear functionals
${\mathcal{U}},$ satisfying ~\eqref{cotra1}, which is defined from the sequence of moments $\{{\mathcal{U}}_m\}_{m \in {\mathbb{N}}}$ using
\begin{equation} \label{xxx}
 {\mathcal{U}}\left({\mathcal{B}}_{0}\right) = \Delta_0 \, , \ \
 {\mathcal{U}}\left({\mathcal{B}}_{m}\right) = 0_{N \times N}, \,m\geq 1,
 \end{equation}
where $\Delta_0$ is a non-singular upper triangular matrix.
Since $\{{\mathcal{P}}_j\}_{j \in {\mathbb{N}}}$ with
$${\mathcal{P}}_j(x)=(h(x))^j {\mathcal{P}}_0(x)\quad \mbox{and} \quad {\mathcal{P}}_0(x)=\left[1\,x\, \cdots\,x^{N-1}\right]^T,$$
is a basis for $\mathbb{P}^N$, then there exists a unique family of matrices $\gamma_j^m \in {\mathcal{M}}_{N \times N}({\mathbb{R}})$ such that the
vector of polynomials ${\mathcal{B}}_m$ can be written ${\mathcal{B}}_m(x) = \sum_{j=0}^m \gamma_j^m {\mathcal{P}}_j(x)$.
Thus,
 \begin{itemize}
\item For $m=0,$
${\mathcal{U}}({\mathcal{B}}_{0})=\gamma_0^0{\mathcal{U}}({\mathcal{P}}_{0})$,
i.e., ${\mathcal{U}}_0=(\gamma_0^0)^{-1}\Delta_0$.
\item For $m=1,$ $\displaystyle{\mathcal{U}}( {\mathcal{B}}_1)=\sum_{j=0}^1 \gamma_j^1{\mathcal{U}}({\mathcal{P}}_{j})$, i.e.,
 ${\mathcal{U}}_1=-(\gamma_1^1)^{-1}\gamma_0^1{\mathcal{U}}_0 $.
\item For $m=2,$ $\displaystyle{\mathcal{U}}( {\mathcal{B}}_2)=\sum_{j=0}^2 \gamma_j^2{\mathcal{U}}({\mathcal{P}}_{j})$, i.e.,
 $\displaystyle {\mathcal{U}}_2=-\sum_{j=0}^1(\gamma_2^2)^{-1}\gamma_j^2{\mathcal{U}}_j $.
\end{itemize}
For $m\geq 3$, we have
$\displaystyle {\mathcal{U}}_m=-\sum_{j=0}^{m-1} (\gamma_m^m)^{-1} \gamma_j^m {\mathcal{U}}_j \, $.

$\phantom{ola}$First, we will prove that for  $\mathcal{U}$ defined as above, we have
 \begin{eqnarray*}
 (h^k {\mathcal{U}} )\left({\mathcal{B}}_m\right) =
0_{N \times N}\, , \quad m \geq k+1.
\end{eqnarray*}
To prove it, we apply  $\mathcal{U}$ in the recurrence relation:
\begin{eqnarray*}
{\mathcal{U}} \left(h{\mathcal{B}}_{m}
\right) = A_{m} {\mathcal{U}} \left( {\mathcal{B}}_{m+1} \right)+
B_m {\mathcal{U}} \left( {\mathcal{B}}_{m} \right) + C_m
{\mathcal{U}} \left( {\mathcal{B}}_{m-1} \right) = 0_{N \times
N} \, , \quad m\geq 2 \, .
\end{eqnarray*}
Again, if we multiply both sides of the recurrence relation by $h$, then we obtain
\begin{eqnarray*}
h^2(x){\mathcal{B}}_m(x) &=& h(x) A_{m} {\mathcal{B}}_{m+1}(x) + h(x) B_m
{\mathcal{B}}_m(x) + h(x) C_m {\mathcal{B}}_{m-1}(x),
\end{eqnarray*}
and, as a consequence, by applying
${\mathcal{U}}$ in the last relation, we have for $m\geq 3$
\begin{equation*} \left(h^2{\mathcal{U}}\right)
\left({\mathcal{B}}_{m} \right) = A_{m}
\left(h{\mathcal{U}}\right) \left({\mathcal{B}}_{m+1} \right)+ B_m
\left(h{\mathcal{U}}\right) \left({\mathcal{B}}_{m} \right) + C_m
\left(h{\mathcal{U}}\right) \left({\mathcal{B}}_{m-1} \right)
= 0_{N \times N} \, .
\end{equation*}
Proceeding in a similar way we get
 $$ (h^k
{\mathcal{U}} )\left({\mathcal{B}}_m\right) =0_{N \times N}\, , \ m\geq k+1,
\mbox{ i.e., }
(h^k {\mathcal{U}})
\left({\mathcal{B}}_m\right) =0_{N \times N}\, , \
k=0,1, \ldots ,m-1.
 $$
For $k=m,$ we have
\begin{equation*}
(h^m {\mathcal{U}}) \left({\mathcal{B}}_{m}
\right) = A_{m} (h^{m-1}{\mathcal{U}}) \left({\mathcal{B}}_{m+1}
\right)+ B_m (h^{m-1}{\mathcal{U}}) \left({\mathcal{B}}_{m} \right)
+ C_m (h^{m-1}{\mathcal{U}}) \left({\mathcal{B}}_{m-1} \right),
\end{equation*}
and so
\begin{equation*}
(h^m {\mathcal{U}})
\left({\mathcal{B}}_{m} \right) =  C_m (h^{m-1}{\mathcal{U}})
\left({\mathcal{B}}_{m-1} \right)
 =  C_m C_{m-1} \, \cdots \,C_1 \Delta_0 \, , \quad m \geq 1 \, .
\end{equation*}
Therefore, the moments associated with the vector of linear functionals $\mathcal{U}$ are uniquely determined from~\eqref{xxx}. Thus, we obtain the orthogonality conditions~\eqref{cotra1}. Hence, the result follows.
\end{proof}
$\phantom{ola}$Next we will introduce the concept of right-orthogonality with respect to a vector of linear functionals and, afterwards, we will show how the right and left vector orthogonality are connected.
\begin{defi}
Let ${\mathcal{U}}=\left[u^{1} \cdots \,
u^{N} \right] ^{T}$ be a vector of linear functionals and let consider a sequence of matrix polynomials $\{G_m\}_{m \in {\mathbb{N}}}$. $\{G_m\}_{m \in {\mathbb{N}}}$ is said to be {\it right-orthogonal} with respect to the vector of linear functionals ${\mathcal{U}}$ if
\begin{itemize}
\item[(a)] $\deg G_{m} = m $.
\item[(b)] $(G_m^T(h(x)) {\mathcal{U}}) \left( {\mathcal{P}}_j \right)= 0_{N \times
N}$, $j=0,1, \ldots ,m-1$.
\item[(c)] $(G_m^T(h(x)) {\mathcal{U}}) \left( {\mathcal{P}}_m \right)=
\Theta_m$, $m \in \mathbb{N},$ where $\Theta_m$ is a non-singular lower triangular matrix.
\end{itemize}
\end{defi}
$\phantom{ola}$Concerning  the right-orthogonality we obtain some analog results to those we found for left-orthogonality.
For example, the matrix right-orthogonal polynomial sequence is uniquely defined up to a multiplication on the right by a non-singular matrix and the vector of linear functionals ${\mathcal{U}}$ is quasi-definite with respect to the right-orthogonality. We will present these results but we shall skip the proofs since the techniques are the same as used in the left-orthogonality case.

%
\begin{teo}
Let $\mathcal{U}$ be a vector of linear functionals. Then, $\mathcal{U}$ is quasi-definite if and only if
there exists a sequence matrix polynomials $\{G_m\}_{m \in {\mathbb{N}}}$, with $G_m (h(x))= \sum_{j=0}^m \beta_j^m (h(x))^j,$ for $
\beta_j^m \in {\mathcal{M}}_{N \times N}({\mathbb{R}})$ where $\beta_m^m$ is non-singular upper triangular matrix and there exists a sequence of non-singular lower triangular matrices, $\{\Theta_m\}_{m \in {\mathbb{N}}}$, such that
$$(G_m^T(h(x)) {\mathcal{U}}) \left( {\mathcal{P}}_m \right)=
\Theta_m\delta_{j,m},\, j=0,1, \ldots ,m-1,\,\, m \in {\mathbb{N}}.$$
Moreover,
\begin{eqnarray}\label{imp2}G_m=\left[
                                  \begin{array}{cccc}
                                    1 & h & \cdots & h^m \\
                                  \end{array}
                                \right]
\left[%
\begin{array}{cccc}
  {\mathcal{U}}({\mathcal{P}}_0)   & \cdots & {\mathcal{U}}({\mathcal{P}}_m)\\
  \vdots  &  \ddots & \vdots \\
  {\mathcal{U}}({\mathcal{P}}_m)   & \cdots & {\mathcal{U}}({\mathcal{P}}_{2m}) \\
\end{array}%
\right]^{-1}\left[\begin{array}{c}
         0 \\ \vdots \\ \Theta_m
         \end{array}\right].
\end{eqnarray}
\end{teo}
\begin{teo}\label{favard2} Let  $\mathcal{U}$ be a vector of linear functionals and let $\{G_m\}_{m \in {\mathbb{N}}}$ be a sequence of matrix polynomials. Then, the following statements are equivalent:
\begin{itemize}
\item[(a)] The  sequence of matrix polynomials $\{G_m\}_{m \in {\mathbb{N}}}$ is right-orthogonal with respect to the vector of linear functionals ${\mathcal{U}}$, i.e.,
    \begin{eqnarray*}
(G_m^T(h(x)) {\mathcal{U}}) \left(
{\mathcal{P}}_j \right)&=& 0_{N \times N}\, , \quad j=0,1, \ldots ,m-1, \\
( G_m^T(h(x)){\mathcal{U}}) \left( {\mathcal{P}}_m \right)&=& \Theta_m
\, , \quad m \in \mathbb{N},
\end{eqnarray*}
where $\Theta_m$ is a non-singular lower triangular matrix.
\item[(b)] There exist sequences of \, $N\times N$ \, matrices, \, $\{D_m\}_{m \in {\mathbb{N}}},$ \, $\{E_m\}_{m\in {\mathbb{N}}}, $ \, and \, $\{F_m\}_{m \in {\mathbb{N}}}$ \, with \, $F_m$ a non-singular lower triangular matrix, such that
\begin{eqnarray*}
\label{recorr1g} h(x) G_m(h(x)) =  G_{m+1}(h(x))D_m + G_m(h(x)) E_m
+  G_{m-1}(h(x))F_m, \ \ m \geq 1 \, ,
\end{eqnarray*}
with
$\displaystyle G_{-1}(x)=0_{N \times N} \ \mbox{ and } \ G_{0}(x)=I_{N \times
N} \, $.
\end{itemize}
\end{teo}
$\phantom{ola}$To show the connection between right and left vector orthogonality we will introduce some concepts on \textit{duality theory}.
We denote by $\mathbb{P}^{\ast}$ the {\it dual
space} of~$\mathbb{P} $, i.e., the vector space of complex valued linear functionals defined on $\mathbb{P}$.

$\phantom{ola}$Let $\{p_{m}\}_{m \in {\mathbb{N}}}$ be a sequence of scalar monic polynomials. The
\textit{sequence of linear functionals} $ \{ L_{n}\}_{n \in {\mathbb{N}}}$,
\textit{where} $ L_{n}\in \mathbb{P}^{\ast }$ is said to be its
\textit{dual sequence}~if
$ L_{n}( p_{m})=\delta_{m,n},\,\ m,n\in \mathbb{N} $, where $\delta_{n,m}$
 is the {\it Kronecker delta}.

$\phantom{ola}$Let $ \{L_{n} \}_{n \in {\mathbb{N}}}$ be a sequence of linear functionals. The vector sequence of linear functionals
$\{\mathcal{L} _{n}\}_{n \in {\mathbb{N}}}$ given by
\begin{equation*}
\mathcal{L}_{n}= \left[
L_{nN} \, \cdots \, L_{(n+1)N-1}
\right] ^{T}, \,\ n\in \mathbb{N} \, , 
\end{equation*}
is said to be the  \textit{vector sequence of linear functionals}
associated with $ \{ L_{n}\}_{n \in {\mathbb{N}}}.$

$\phantom{ola}$Taking into account Definition~\ref{vec_fun_lin}, we get
\begin{equation*}
 \mathcal{L}_{n}(\mathcal{B}_{m})=\left[
\begin{matrix}
L_{nN} (p_{mN}) &  \cdots  &
 L_{(n+1)
N-1}(p_{mN})  \\
\vdots &  \ddots  &  \vdots \\
 L_{nN}(p_{(m+1)N-1}) & \cdots & L_{(n+1)N-1}(p_{
(m+1)N-1} )
\end{matrix}
\right]
 = I_{N \times N} \delta_{m,n} \, .
 \end{equation*}
 \begin{defi}
Let $\{\mathcal{B}_{m}
\}_{m \in {\mathbb{N}}}$ be a vector sequence of polynomials. The \textit{vector
sequence of linear functionals}
$\{\mathcal{L}_{n}\}_{n \in {\mathbb{N}}}$ is said to be its
\textit{dual vector sequence} ~if
 \begin{equation*}
\mathcal{L}_{n}(\mathcal{B}_{m})=I_{N\times N} \, \delta _{m,n} \,
, \ \ n , m \in \mathbb{N} \, .
 \end{equation*}
 \end{defi}
$\phantom{ola}$The next two results give the connection between right and left vector orthogonality through the equivalent conditions of these two types of vector orthogonality.
 \begin{defi}
Let ${\mathcal{U}}$ be a vector of linear functionals. We denote by
$\widehat{{\mathcal{U}}},$ the {\it nor\-ma\-li\-zed vector of linear functionals}
associated with ${\mathcal{U}},$
$\displaystyle \widehat{{\mathcal{U}}}=(({\mathcal{U}}({\mathcal{P}}_0))^{-1})^T {\mathcal{U}} \, ,$
where ${\mathcal{P}}_0(x) =\left[1, \,
x, \, \cdots, \, x^{N-1}\right]^T.$
\end{defi}
Furthermore, from this definition we have
\begin{eqnarray*}
\widehat{{\mathcal{U}}}({\mathcal{P}}_0) = ((({\mathcal{U}}({\mathcal{P}}_0))^{-1})^T
{\mathcal{U}})({\mathcal{P}}_0) = {\mathcal{U}}({\mathcal{P}}_0)({\mathcal{U}}({\mathcal{P}}_0))^{-1} = I_{N \times N}.
\end{eqnarray*}
 \begin{teo} \label{osvm}
 Let ${\mathcal{U}}$ be a quasi-definite vector of linear functionals, $\{{\mathcal{B}}_m\}_{m \in {\mathbb{N}}}$ be a sequence of vector polynomials, and let $\{{\mathcal{L}}_n\}_{n \in {\mathbb{N}}}$ be  its dual vector sequence. Then, the
following statements are equivalent:
\begin{itemize}
\item[(a)]$\{{\mathcal{B}}_m\}_{m \in {\mathbb{N}}}$ is left-orthogonal with respect to  ${\mathcal{U}}$.
\item[(b)]There exist sequences of~$N\times N$ matrices, $ \{A_m\}_{m \in {\mathbb{N}}}, $ $\{B_m\}_{m \in {\mathbb{N}}}, $ \,  and \, $ \{C_m\}_{m \in {\mathbb{N}}}  $ \, with  \, $  C_m $ a non-singular upper triangular matrix, such that
$\{{\mathcal{B}}_m\}_{m \in {\mathbb{N}}}$ satisfies the three-term recurrence relation
\begin{eqnarray} \label{rr1} h(x)
{\mathcal{B}}_m(x) = A_{m} {\mathcal{B}}_{m+1}(x) + B_m
{\mathcal{B}}_m(x) + C_m {\mathcal{B}}_{m-1}(x), \quad m\geq 1
\end{eqnarray}
with
$\displaystyle {\mathcal{B}}_{-1}(x) = 0_{1 \times N}$ and $\displaystyle {\mathcal{B}}_{0}(x)={\mathcal{P}}_0(x) \, $,
where $\displaystyle{\mathcal{P}}_0(x)=\left[1\, x\, \cdots \, x^{N-1}\right]^T$.
\item[(c)]There exist sequences of~$N\times N$ matrices, $\{A_n\}_{\in {\mathbb{N}}},$ $\{B_n\}_{n \in {\mathbb{N}}},$ and $\{C_n\}_{n \in {\mathbb{N}}},$ with $C_{n+1}$
a non-singular matrix, such that $\{{\mathcal{L}}_n\}_{n \in {\mathbb{N}}}$ is defined by the three-term recurrence relation
 \begin{eqnarray} \label{recorr4}
 h(x)
{\mathcal{L}}_n = (C_{n+1})^T {\mathcal{L}}_{n+1} + (B_n)^T
{\mathcal{L}}_n + (A_{n-1})^T {\mathcal{L}}_{n-1}, \quad n\geq 1,
\end{eqnarray}
where
$\displaystyle {\mathcal{L}}_0=(({\mathcal{U}}({\mathcal{P}}_0))^T)^{-1}{\mathcal{U}},
\ {\mathcal{L}}_1 = (C_1^T)^{-1}(h(x)I
-(B_0)^T)[{\mathcal{U}}({\mathcal{P}}_0))^T]^{-1}{\mathcal{U}} \, .$
\item[(d)]There exist matrix polynomials $G_n(h(x))$ with
$\displaystyle G_n(h(x))=\sum_{j=0}^n \beta_j^n (h(x))^j \, ,$
where $\beta_n^n$ is a non-singular matrix, such that the elements in the dual vector basis
$\{{\mathcal{L}}_n \}_{n \in {\mathbb{N}}}$ can be written in terms of the vector of linear functionals ${\mathcal{U}}$ as follows
\begin{eqnarray} \label{batata}
{\mathcal{L}}_n = \left(G_n (h(x)) \right)^T {\mathcal{U}},\quad n
\in {\mathbb{N}}.
\end{eqnarray}
\item[(e)]The sequence of matrix polynomials $\{G_n\}_{n \in {\mathbb{N}}}$ defined by
\eqref{batata} satisfies
\begin{eqnarray} \label{rrrG}
h(x) G_n(h(x)) =
G_{n-1}(h(x))A_{n-1} + G_n(h(x))B_n+ G_{n+1}(h(x))C_{n+1}, \ n\geq 1 \, ,
\end{eqnarray}
with initial conditions $\displaystyle G_{-1}(h(x))=0_{N \times N} \,\, \mbox{and} \,\, G_0(h(x))={\mathcal{U}}({\mathcal{P}}_0)^{-1}\,$.

\item[(f)]The sequence of matrix polynomials $\{G_n\}_{n \in {\mathbb{N}}}$ defined by~\eqref{batata} is right-{\linebreak}orthogonal with respect to the normalized vector of linear functionals ${\mathcal{U}}$.
    \end{itemize}
    \end{teo}
\begin{proof}
We will prove this theorem according to the following scheme:

$\phantom{ola}$$(a) \Leftrightarrow (b)$, $(e) \Leftrightarrow (f)$, $(b) \Leftrightarrow (c)$, $(c) \Rightarrow (d)$, $(d) \Rightarrow (e),$ and $(e) \Rightarrow (c)$.

$\phantom{ola}$The proofs of $(a) \Leftrightarrow (b)$ and  $(e) \Leftrightarrow (f)$  follow immediately from Theorems~\ref{FAVARD1} and~\ref{favard2}.
We start by proving that $(b)\Rightarrow (c)$.
Let
\begin{eqnarray*} h(x){\mathcal{L}}_n = \sum_{j=0}^{n+1}
\beta_j^n {\mathcal{L}}_j, \quad \mbox{where} \quad
(\beta_j^n)^T=(h(x){\mathcal{L}}_n)({\mathcal{B}}_j)={\mathcal{L}}_n(h(x){\mathcal{B}}_j), \,
j \in {\mathbb{N}}.
\end{eqnarray*}
Applying the vector of linear functionals ${\mathcal{L}}_n$ in both sides of the three-term recurrence relation satisfied by $\{{\mathcal{B}}_k\}_{k \in {\mathbb{N}}}$, we have
\begin{eqnarray*}
(\beta_j^n)^T &=& A_k{\mathcal{L}}_n({\mathcal{B}}_{k+1})+B_k{\mathcal{L}}_n({\mathcal{B}}_{k})+C_k{\mathcal{L}}_n({\mathcal{B}}_{k-1}))\\
 & = & \left\{%
\begin{array}{ll}
    A_{n-1}, & j=n-1, \\
    B_n, &  j=n, \\
    C_{n+1}, & j=n+1, \\
    0_{N\times N}, & j \neq n-1,n,n+1, \\
\end{array}%
\right.
\end{eqnarray*}
i.e., $\beta_{n-1}^n = A_{n-1}^T$, $\beta_{n}^n =
B_{n}^T,$ and $\beta_{n+1}^n = C_{n+1}^T.$ Thus the three-term recurrence relation for the vector sequence of linear functionals $\{{\mathcal{L}}_n\}_{n \in {\mathbb{N}}}$ follows.

$\phantom{ola}$To prove that $(c)\Rightarrow(b)$, let $ \displaystyle
h{\mathcal{B}}_m=\sum_{j=0}^{m+1} \gamma_j^m {\mathcal{B}}_j, \quad
\gamma_j^m \in {\mathcal{M}}_{N \times N}({\mathbb{R}})$. Applying the vector linear functional
${\mathcal{L}}_n$ in both sides of the last relation, we get $\displaystyle
h{\mathcal{L}}_n({\mathcal{B}}_m)=\sum_{j=0}^{m+1} \gamma_j^m
{\mathcal{L}}_n({\mathcal{B}}_j)= \gamma_n^m$. Now, from our hypotheses, we have
\begin{eqnarray*}
 \gamma_n^m
 & = &
{\mathcal{L}}_{n+1}({\mathcal{B}}_m)C_{n+1} +
{\mathcal{L}}_n({\mathcal{B}}_m)B_n +
{\mathcal{L}}_{n-1})({\mathcal{B}}_m)A_{n-1}\\&=&\left\{%
\begin{array}{ll}
    C_{m}, & n=m-1, \\
    B_n, &  n=m ,\\
    A_{m}, & n=m+1,\\
    0_{N\times N}, & n \neq m-1,m,m+1. \\
\end{array}%
\right.
\end{eqnarray*} So,~\eqref{rr1} holds.

$\phantom{ola}$To prove $(c)\Rightarrow(d)$  we will show by induction that $\{{\mathcal{L}}_n\}_{n \in {\mathbb{N}}}$ has the following representation  $\displaystyle
{\mathcal{L}}_n=\left(G_n (h(x)) \right)^T {\mathcal{U}},\, n \in
{\mathbb{N}}.$ For $n=0$, we have that $\displaystyle
{\mathcal{L}}_0=(({\mathcal{U}}({\mathcal{P}}_0))^T)^{-1}{\mathcal{U}}$. Now, let us assume that the statement holds for $k=0,1,\ldots,
p$, i.e., $ \displaystyle {\mathcal{L}}_k=\left(G_k (h(x))
\right)^T {\mathcal{U}}$ with $\deg G_k=k$, $k=1,\ldots, p.$ We will show that it is also true for
$k=p+1$, i.e., $\displaystyle {\mathcal{L}}_{p+1}=\left(G_{p+1} (h(x))
\right)^T {\mathcal{U}}, \, p \in {\mathbb{N}}$. Considering the three-term recurrence relation
satisfied by $\{{\mathcal{L}}_{p}\}_{p \in {\mathbb{N}}}$ and taking into account the hypothesis of induction, we have
 \begin{eqnarray*}
 {\mathcal{L}}_{p+1}&=&(C_{p+1})^{-T}\left[(h(x)I-(B_p^T))G_p(h(x))-(A_{p-1})^TG_{p-1}(h(x))\right]{\mathcal{U}}\\
&=&\left[(G_p(h(x))^T((h(x)I-(B_p))-G_{p-1}^T(h(x))A_{p-1}))C_{p+1}^{-1}\right]^T{\mathcal{U}}.
\end{eqnarray*}
Thus, $\displaystyle {\mathcal{L}}_{p+1}=\left(G_{p+1} (h(x))
\right)^T {\mathcal{U}}, \, p \in {\mathbb{N}}$, i.e., if the condition holds for $k=1,\ldots,p,$ then it is also true for $p+1$.

$\phantom{ola}$To prove that $(d)\Rightarrow(e)$ we will write $hG_n^T$ in terms of $\{G_j^T\}_{j \in {\mathbb{N}}}$, i.e.,
\begin{eqnarray}
\label{rrG} h(x)G_n^T (h(x))= \sum_{j=0}^{n+1} \alpha_j^n
G_j^T(h(x)), \quad \mbox{where} \quad \alpha_j^n \in \quad
{\mathcal{M}}_{N \times N} ({\mathbb{R}}).
\end{eqnarray}
Thus, the multiplication on the right by ${\mathcal{U}}$ in both sides of the last equation yields
$$h(x)G_n^T (h(x)){\mathcal{U}}=
\sum_{j=0}^{n+1} \alpha_j^n G_j^T(h(x)){\mathcal{U}}.$$ Applying this relation to  ${\mathcal{B}}_k$ we get
$$(h(x)G_n^T (h(x)){\mathcal{U}})({\mathcal{B}}_k)=
\sum_{j=0}^{n+1} (\alpha_j^n
G_j^T(h(x)){\mathcal{U}})({\mathcal{B}}_k).$$
Since, ${\mathcal{L}}_n = \left(G_n (h(x)) \right)^T {\mathcal{U}},\quad n
\in {\mathbb{N}},$
\begin{eqnarray*}{\mathcal{L}}_n\left(h(x){\mathcal{B}}_k\right) = \sum_{j=0}^{n+1}
{\mathcal{L}}_j\left({\mathcal{B}}_k\right)(\alpha_j^n)^T 
 = (\alpha_k^n)^T .
\end{eqnarray*}
Using~(\ref{rr1}) in~(\ref{rrG}) we get
\begin{eqnarray*}
(
\alpha_k^n)^T&=& C_k {\mathcal{L}}_n \left( {\mathcal{B}}_{k-1} \right) + B_k {\mathcal{L}}_n \left({\mathcal{B}}_{k} \right)
+ A_k {\mathcal{L}}_n \left( {\mathcal{B}}_{k+1} \right) \\
 & = & \left\{%
 \begin{array}{ll}
    A_{n-1}, & k=n-1, \\
    B_n, &  k=n, \\
    C_{n+1}, & k=n+1, \\
    0_{N\times N}, & k \neq n-1,n,n+1. \\
 \end{array}%
\right.
\end{eqnarray*}
Thus, $\{G_n\}_{n \in {\mathbb{N}}}$ satisfies
$$h(x)G_n(h(x)) =  G_{n-1}(h(x))A_{n-1} +  G_n(h(x))B_n+ G_{n+1}(h(x))C_{n+1}.$$

$\phantom{ola}$Finally, to prove that $(e)\Rightarrow(c)$, we must take the transpose in the recurrence relation
\eqref{rrrG} and, then, multiply on the right by ${\mathcal{U}}$ both sides of the resulting equation. Thus ~\eqref{recorr4} follows.
\end{proof}
\begin{teo}\label{teobiort}
Let ${\mathcal{U}}$ be a quasi-definite vector of linear functionals, $\{{\mathcal{B}}_m\}_{m \in {\mathbb{N}}}$ be a sequence of vector polynomials and $\{G_n\}_{n \in {\mathbb{N}}}$ defined by~\eqref{imp2}.
Then, $\{{\mathcal{B}}_m\}_{m \in {\mathbb{N}}}$ and $\{G_n\}_{n \in {\mathbb{N}}}$ are bi-orthogonal with respect to ${\mathcal{U}}$, i.e.,
$$((G_n(h(x)))^T{\mathcal{U}})({\mathcal{B}}_m)=I_{N \times N} \delta_{n,m}, n,\, m \in {\mathbb{N}},$$
if and only if  $\Delta_m =(\beta_m^m)^{-1}$  and $\Omega_n=(\alpha_m^m)^{-1}$. \\
As a consequence, the dual sequence $\{{\mathcal{L}}_n\}_{n \in {\mathbb{N}}}$ associated with
$\{{\mathcal{B}}_m\}_{m \in {\mathbb{N}}}$ is given by ${\mathcal{L}}_n=(G_n(h(x)))^T{\mathcal{\mathcal{U}}},$~$n \in {\mathbb{N}}$.
\end{teo}
\begin{proof}
There exists a unique family of matrices $(\alpha_j^m) \subset {\mathcal{M}}_{N \times N}({\mathbb{R}})$ such that
${\mathcal{B}}_m=\sum_{j=0}^m \alpha_j^m{\mathcal{P}}_j$, where $\alpha_m^m$ is a non-singular matrix. Hence,
$$(G_n^T(h(x)){\mathcal{\mathcal{U}}})({\mathcal{B}}_m)=(G_n^T(h(x)){\mathcal{\mathcal{U}}})(\sum_{j=0}^m \alpha_j^m{\mathcal{P}}_j)=\sum_{j=0}^m \alpha_j^m (G_n^T(h(x)){\mathcal{\mathcal{U}}})({\mathcal{P}}_j).$$
Since $\{G_n\}_{n \in {\mathbb{N}}}$ is right-orthogonal with respect to ${\mathcal{U}}$ then
$$(G_n^T(h(x)){\mathcal{\mathcal{U}}})({\mathcal{B}}_m)=\left\{
                                                          \begin{array}{ll}
                                                            \alpha_m^m \Theta_m, & m=n \\
                                                             0_{N \times N}, & m>n.
                                                          \end{array}
                                                        \right.
$$
Thus, $(G_m^T(h(x)){\mathcal{\mathcal{U}}})({\mathcal{B}}_m)=I_{N \times N}$ if and only if $\alpha_m^m \Theta_m=I_{N \times N}$, i.e.,
$\Theta_m=(\alpha_m^m)^{-1}$. Now, let us consider
$$(G_n^T(h(x)){\mathcal{\mathcal{U}}})({\mathcal{B}}_m)=((\sum_{j=0}^n \beta_j^n (h(x))^j)^T{\mathcal{\mathcal{U}}})({\mathcal{B}}_m)=\sum_{j=0}^n( h(x)^j{\mathcal{\mathcal{U}}})({\mathcal{B}}_m)\beta_j^n.$$
As above, since $\{{\mathcal{B}}_m\}_{m \in {\mathbb{N}}}$ is left-orthogonal with respect to ${\mathcal{U}}$ then
$$(G_n^T(h(x)){\mathcal{\mathcal{U}}})({\mathcal{B}}_m)=\left\{
                                                          \begin{array}{ll}
                                                            \Delta_m \beta_m^m, & m=n \\
                                                             0_{N \times N}, & m>n.
                                                          \end{array}
                                                        \right.$$
So, $(G_m^T(h(x)){\mathcal{\mathcal{U}}})({\mathcal{B}}_m)=I_{N \times N}$ if and only if 
$\Delta_m=(\beta_m^m)^{-1}.$
\end{proof}
$\phantom{ola}$To conclude this section notice that Theorems~\ref{teorema2} and~\ref{osvm} suggest that the sequence of matrix polynomials $\{V_m\}_{m \in {\mathbb{N}}}$ is orthogonal with respect to some matrix of measures. As a consequence of Theorem~\ref{osvm}, the sequence of matrix polynomials $\{G_m\}_{m \in {\mathbb{N}}}$ should also be orthogonal in the matrix sense.
Finally, Theorem~\ref{teobiort} suggests that the sequences of matrix polynomials $\{V_m\}_{m \in {\mathbb{N}}}$ and $\{G_m\}_{m \in {\mathbb{N}}}$ should be bi-orthogonal to each order, as we will prove in the next section.

\section{The connection between vector and matrix orthogonality}

$\phantom{ola}$In this section we show how the vector and matrix orthogonality are connected when a special case of the matrix orthogonality is considered.
In this sense, if the sequences $\{G_n\}_{n \in {\mathbb{N}}}$ and $\{{\mathcal{B}}_m\}_{m \in {\mathbb{N}}}$ are bi-orthogonal with respect to the vector of linear functionals ${\mathcal{U}},$ then the sequences of matrix polynomials $\{G_n\}_{n \in {\mathbb{N}}}$ and $\{V_m\}_{m \in {\mathbb{N}}}$ are bi-orthogonal with respect to a complex matrix of measures.
 \begin{defi}
Let ${\mathcal{U}}$ be a vector of linear functionals. We define the {\it generalized Markov
matrix function}, ${\mathcal{F}}$, associated with ${\mathcal{U}}$ by
 \begin{eqnarray}\label{fmarkov}
{\mathcal{F}}(z) :=
{\mathcal{U}}_x \left( \frac{{\mathcal{P}}_{0}(x)}{z-h(x)}\right) =
\left[
 \begin{array}{cccc}
\langle u^1_x , \frac{1}{z-h(x)}\rangle  & \cdots & \langle u^N_x , \frac{1}{z-h(x)}\rangle  \\
 \vdots & \ddots & \vdots \\
\langle u^1_x , \frac{x^{N-1}}{z-h(x)}\rangle & \cdots & \langle u^N_x , \frac{x^{N-1}}{z-h(x)}\rangle \\
\end{array}%
\right],
\end{eqnarray}
with $z$ such that $|h(x)|<|z|$ for every  $x \in {\sf L}$ where
$\displaystyle {\sf L} =\displaystyle \cup_{j=1,\ldots,N} \,\mbox{supp} \, u^j_x \, $.
Here ${\mathcal{U}}_x$ represents the action of
${\mathcal{U}}$ on the variable $x$ and ${\mathcal{P}}_0(x)=\left[1\,
x\, \cdots \, x^{N-1}\right]^T$.
\end{defi}
\begin{teo}\label{medidacomplexa}
Let $ \mathcal{U}$ be a quasi-definite vector of linear
functionals and let $ \mathcal{F}$ be its generalized Markov matrix function.
Then, the following statements
are equivalent: \\
$\phantom{ll}$a) The sequences $ \{ G_{n}\}_{n \in {\mathbb{N}}}$ and $ \{ \mathcal{B}_{m} \}_{m \in {\mathbb{N}}}$ are bi-orthogonal with respect to $\mathcal{U}$, i.e.
 \begin{equation*}
 (   (  G_{n} (  h(x) )   ) ^{T}\mathcal{U}_{x} )
 (  \mathcal{B}_{m} )  =I_{N\times N} \, \delta_{n,m},\text{\hspace{
0.1in}}n,m\in \mathbb{N} \, .
 \end{equation*}
$ \phantom{ll}$b) The sequences $ \{ G_{n} \}_{n \in {\mathbb{N}}}$ and $ \{ V_{m} \}_{m \in {\mathbb{N}}}
$, where
 $\mathcal{B}_m (z) = V_m (h(z)) \mathcal{P}_0(z)$,
 are bi-orthogonal with respect to $\mathcal{F}$, i.e.,
 \begin{equation*}
\frac{1}{2\pi i}\int_{C}V_{m} (  z )  \mathcal{F} ( z )  G_{n}
( z )  dz=I_{N \times N} \, \delta _{n,m},\text{\hspace{
0.1in}}n,m\in \mathbb{N} \, .
 \end{equation*}
where $C$ is a closed path in $\{z \in {\mathbb{C}}: |z| > |h(x)|, x \in {\sf L} \}$.
\end{teo}
\begin{proof}
Taking into account that
 \begin{equation*}
 V_{m} (  z )  \mathcal{F} (  z )  G_{n} ( z )
= \left(  \frac{V_{m} (  z ) \mathcal{P}_{0} ( x )
}{z-h(x)} \right) \, \mathcal{U}_{x}^{T}G_{n} ( z )
 = \left(   (  G_{n} (  z )   ) ^{T}\mathcal{U} _{x} \right)   \left(
\frac{V_{m} ( z )  \mathcal{P}_{0} ( x )  }{z-h(x)} \right)  \, ,
 \end{equation*}
we have
 \begin{equation*}
\frac{1}{2\pi i}\int_{C}V_{m} (  z )
\mathcal{F} ( z ) G_{n} (
z )  dz = \frac{1}{2\pi i}\int_{C} \left( ( G_{n} ( z ) )
^{T}\mathcal{U}_{x} \right)   \left( \frac{V_{m} ( z ) \mathcal{P}_{0} ( x
)  }{z-h(x)} \right) dz \, .
 \end{equation*}
Because of $ G_{n},$ $ V_{m},$ and
$ \mathcal{P}_{0}$ are analytic functions,
according to the Cauchy integral formula we have
 \begin{equation*}
 \frac{1}{2\pi i}\int_{C} \left(   (  G_{n} ( z )  )
^{T}\mathcal{U}_{x} \right)   \left(  \frac{V_{m} (  z )  \mathcal{P} _{0}
(  x )  }{z-h(x)} \right)  dz = (   ( G_{n} ( h(x) ) ) ^{T}\mathcal{U
}_{x} )   (  V_{m} (  h(x) ) \mathcal{P}_{0} ( x )
 ),
 \end{equation*}
and, as a consequence, for all $n,m\in \mathbb{N}$
 \begin{equation*}
\frac{1}{2\pi i}\int_{C}V_{m} (  z ) \mathcal{F} ( z )  G_{n} (
z )  dz = (   ( G_{n} (  h(x) ) ) ^{T}\mathcal{U }_{x} )   ( \mathcal{B}_{m}
( x ) ) = I_{N\times N} \, \delta_{n,m} \, .
 \end{equation*}
 Thus the statement follows.
\end{proof}

$\phantom{ola}$The last theorem tell us that {\it $\{{\mathcal{B}}_m\}_{m \in {\mathbb{N}}}$ is a sequence of vector polynomials left-orthogonal with respect to ${\mathcal{U}}$ } if and only if {\it $\{V_m\}_{m \in {\mathbb{N}}}$ associated with $\{{\mathcal{B}}_m\}_{m \in {\mathbb{N}}}$ is left-orthogonal with respect to ${\mathcal{F}}$}. Also, {\it $\{G_m\}_{m \in {\mathbb{N}}}$ is a sequence matrix polynomials right-orthogonal with respect to ${\mathcal{U}}$ } if and only if {\it $\{G_m\}_{m \in {\mathbb{N}}}$ is right-orthogonal with respect to ${\mathcal{F}}$}.

$\phantom{ola}$It is important to recall now that the definition of ${\mathcal{F}}$ shows us that we only need~$N$ linear functionals to describe the matrix orthogonality. Usually, to describe the matrix orthogonality, ${(1+N)N}/{2}$ measures are needed (see, for example~\cite{Dur93,Dur95,DurWal95}).

$\phantom{ola}$As we have already referred in the introduction, we need to know when a sequence of matrix polynomials defined by a recurrence relation ~\eqref{rrvm} is related to  some kind of matrix orthogonality. Partial answers were given to this problem, but no complete answer was given as far as we know. To do that, we start by considering a~$N\times N$ matrix of measures $W$ that is not necessarily positive definite  in ${\mathbb{R}},$ and such that there exist matrix sequences $\{V_m\}_{m \in {\mathbb{N}}}$ and $\{G_m\}_{m \in {\mathbb{N}}}$,
orthogonal with respect to $W$ in the following sense
\begin{eqnarray}\label{ortmatricialesq}
\int_{{\mathbb{R}}} V_m(x) dW(x) x^k & = &
 \Omega^1_m \delta_{k,m}, \quad k,m\geq 0,
 \\ \label{ortmatricialdir}
\int_{{\mathbb{R}}} x^k  dW(x) G_m(x) & = &
 \Omega^2_m \delta_{k,m}, \quad k,m\geq 0,
 \end{eqnarray}
where $\Omega^1_m$ is a non-singular upper triangular matrix, $\Omega^2_m$ is a non-singular lower triangular matrix, and $\delta_{k,m}$ is the Kronecker delta.

$\phantom{ola}$$V_m$ and $G_m$ are matrix polynomials of degree $m$ with non-singular leading coefficients and they are defined up to the multiplication on the left or on the right by a unitary matrix, respectively. The matrix sequences $\{V_m\}_{m \in {\mathbb{N}}}$ (respectively, $\{G_m\}_{m \in {\mathbb{N}}}$) satisfying~\eqref{ortmatricialesq} (respectively,~\eqref{ortmatricialdir}) are said to be the {\it left-orthogonal matrix polynomial sequence} (respectively, {\it right-orthogonal matrix polynomial sequence}), with respect to the matrix of measures ~$W$.

$\phantom{ola}$Usually, in the theory of matrix orthogonal polynomials there are only references to the left-orthogonality. The reason is that the authors deal only with orthonormality with respect to a positive definite matrix of measures, that allows us to say that left and right orthogonality are, essentially, the same. Very few authors have emphasized this difference (see, for instance, ~\cite{simon}).

$\phantom{ola}$The {\it moments} of the matrix measure $W$ are given by~$N\times N$ matrices
\begin{eqnarray*}
S_k=\int_{{\mathbb{R}}} x^k  dW(x), \quad k=0,1,\ldots\,.
\end{eqnarray*}

$\phantom{ola}$From the orthogonality conditions it follows that the sequences $\{V_m\}_{m \in {\mathbb{N}}}$ and $\{G_m\}_{m \in {\mathbb{N}}}$ satisfy three-term matrix recurrence relations. It is not so obvious to prove the converse result, i.e.  if a sequence of matrix polynomials is defined by a recurrence relation ~\eqref{rrvm} or ~\eqref{rrrG}, then there exists a matrix of measures $W,$ not necessarily positive definite, such this sequence is left or right-orthogonal with respect to $W$.

$\phantom{ola}$The following result proves this equivalence with respect to the left-orthogonality and gives an extension of the Favard's theorem in the matrix case.
\begin{teo}
\label{favardv}
Let $\{V_m\}_{m \in {\mathbb{N}}}$ be a sequence of matrix polynomials. Then, the following statements are equivalent:
\begin{itemize}
\item[(a)] The sequence $\{V_m\}_{m \in {\mathbb{N}}}$ is left-orthogonal with respect to matrix of measures~$W$.
\item[(b)] There are sequences of scalar matrices $\{ A_m \}_{m \in {\mathbb{N}}}$, $\{ B_m \}_{m \in {\mathbb{N}}}$ and $\{ C_m \}_{m \in {\mathbb{N}}}$, with $A_m$ lower-triangular, and $C_{m+1}$ upper-triangular, non-singular matrices for \linebreak $m \in \mathbb{N}$, such that the sequence $\{V_m\}_{m \in {\mathbb{N}}}$ satisfies
 \begin{eqnarray}
 \label{rrvm1}
zV_m(z)=A_m V_{m+1}(z) +B_m V_m(z)+C_{m} V_{m-1}(z),\quad
m\geq 1,
\end{eqnarray}
 where
$\displaystyle V_{-1}(z)=0_{N \times N} \quad \mbox{and} \quad
V_{0}(z)=I_{N \times N} \, $.
\end{itemize}
\end{teo}
\begin{proof}
 First we will prove that (a) implies (b). Since the sequence $\{V_m\}_{m \in {\mathbb{N}}}$ is  a basis in the linear space of matrix polynomials we can write
$$z V_m(z) = \sum_{k=0}^{m+1} A_k^m V_k(z), \quad A^m_k \in {\mathcal{M}}_{N \times N} ({\mathbb{R}}).$$
Then, from the orthogonality conditions, we get
$$A_j^m \int_{{\mathbb{R}}}
V_j(z)dW(z)z^j=\int_{{\mathbb{R}}} V_m(z)dW(z)z^{j+1}=0_{N \times N}
\quad \mbox{for}\quad  j=0, \ldots, m-2.$$
Thus, $$zV_m(z)= A_{m-1}^mV_{m-1}(z) + A_m^m V_m(z) + A_{m+1}^m
V_{m+1}(z),$$ where $$A_m^m = \left(\int_{{\mathbb{R}}}z V_m dW(z)
z^m\right)(\Omega^1_m)^{-1} \, , \
A_{m-1}^m=\left(\int_{{\mathbb{R}}} z V_m  (z) dW(z)
z^{m-1}\right) $$
$$\times (\Omega^1_{m-1})^{-1} \, , \mbox{ and } \ A_{m+1}^m=\left(\int_{{\mathbb{R}}} z V_m (z) dW(z)
z^{m+1} \right)(\Omega^1_{m+1})^{-1}.$$
 Taking $A_m=A_{m+1}^m$, $B_m=
A_m^m,$ and $C_m = A_{m-1}^m$ the result follows.

$\phantom{ola}$Finally, to prove that $(b)$ implies $(a)$, we should start by
defining recursively the matrix moments associated with the matrix of
measures $W$ by the following conditions
$$S_0 = \int_{{\mathbb{R}}} dW(z) = \Omega^1_0 \quad \mbox{and} \quad \int_{{\mathbb{R}}} V_m(z) dW(z)=0_{N \times N}, \, \, m\geq 1,$$
where $\Omega^1_0$ is a non-singular upper triangular matrix.

$\phantom{ola}$Taking into account that $V_m$ can be written
$$V_m(z)= V_{m,m}z^m + \cdots + V_{m,1}z +V_{m,0}$$ with $V_{m,m}$ a non-singular matrix, then we have
\begin{equation*}
0_{N \times N} =  \int_{{\mathbb{R}}} V_m (z) dW(z)
 = V_{m,m} S_m + \cdots + V_{m,0}S_0.
 \end{equation*}
Thus, the moments are defined recursively by
$\displaystyle S_m= V_{m,m}^{-1} \sum_{j=0}^{m-1} V_{m,j} S_j \, .$

$\phantom{ola}$Let us show that
\begin{equation*}
\int_{{\mathbb{R}}}V_m (z)dW(z) z^k
 = 0_{N \times N}, \quad k=0, \ldots, m-1,  \ \mbox{ and } \ \
 \int_{{\mathbb{R}}}V_m (z) dW(z) z^m = \Omega^1_m.
 \end{equation*}
From~\eqref{rrvm1} we get
$\displaystyle \int_{{\mathbb{R}}}zV_m(z) dW(z)
 = 0_{N \times N}, \,\, m\geq 2 \, $,
 Again, by multiplying both sides of the recurrence relation by $z$ we get
$$z^2V_m(z)=A_m zV_{m+1}(z) +B_m zV_m(z)+C_{m} zV_{m-1}(z),$$ and
, as a consequence,
$\displaystyle  \int_{{\mathbb{R}}}V_m(z) dW(z)z^2
= 0_{N \times N}, \quad m\geq 3 \, $.

$\phantom{ola}$In an analog way, we conclude that
$ \displaystyle \int_{{\mathbb{R}}}V_m (z) dW(z) z^k
=0_{N \times N}, m\geq k+1 , $ and so
$\displaystyle \int_{{\mathbb{R}}}V_m (z) dW(z) z^k
=0_{N \times N}, \ k=0, \ldots, m-1 \, . $ \\
For $k=m$ we have
\begin{eqnarray*}
\int_{{\mathbb{R}}}V_m(z) dW(z)z^m
 = C_m \int_{{\mathbb{R}}}V_{m-1}(z)
dW(z)z^{m-1}  = C_m C_{m-1} \cdots C_1\Omega^1_0 .
\end{eqnarray*}
Then
 $$
\int_{{\mathbb{R}}}V_m (z) dW(z) z^k
=0_{N \times N}, \quad k=0, \ldots, m-1 \, \mbox{ and } \
 \int_{{\mathbb{R}}}V_m (z) dW(z) z^m  = \Omega^1_m \, ,
 $$
where
$\Omega^1_m=C_mC_{m-1}\cdots C_1\Omega^1_0$ is a non-singular upper triangular matrix.
\end{proof}
$\phantom{ola}$As an important remark the reader should notice that left vector orthogonality and matrix orthogonality are equivalent.
This equivalence is given by Theorems~\ref{teorema2},~\ref{osvm}, and~\ref{favardv}. A similar result can be obtained for the right-orthogonality.
\begin{teo} \label{favardG}
Let $\{G_m\}_{m \in {\mathbb{N}}}$ be a sequence of matrix polynomials. Then, the following statements are equivalent:
\begin{itemize}
\item[(a)]$\{G_m\}_{m \in {\mathbb{N}}}$ is a right-orthogonal sequence of matrix polynomials with respect to a matrix of measures $W$.
\item[(b)] There are sequences of scalar matrices $\{ A_m \}_{m \in {\mathbb{N}}}$, $\{ B_m \}_{m \in {\mathbb{N}}}$ and $\{ C_m \}_{m \in {\mathbb{N}}}$, with $A_{m-1}$ lower-triangular, and $C_{m+1}$ upper-triangular, non-singular matrices for \linebreak $m \in \mathbb{N}$, such that the sequence $\{G_m\}_{m \in {\mathbb{N}}}$ satisfies
\begin{eqnarray}
 \label{rrgm}
G_m(z) = G_{m-1}(z)A_{m-1} +  G_m(z)B_m+ G_{m+1}(z)C_{m+1},\quad
m\geq 1
\end{eqnarray}
where
$\displaystyle G_{-1}(z)=0_{N \times N} \quad \mbox{and} \quad
G_0(z)=I_{N \times N} \, $.
\end{itemize}
 \end{teo}

\section{Some characterizations of the vector and matrix orthogonality}

$\phantom{ola}$In this section we present some characterizations of the vector orthogonality as well as of the matrix orthogonality. First, we analyze two type Hermite-Pad\'{e} approximation problems that characterize completely the right and left vector orthogonality, respectively.
\begin{defi}
Let $\displaystyle\{\mathcal{B}_{m}\}_{m \in {\mathbb{N}}}$ be a vector sequence of polynomials, let $\{G_{m}\}_{m \in {\mathbb{N}}}$ be a sequence of matrix polynomials, and let $\displaystyle
\mathcal{U}$ be a quasi-definite vector of linear functionals. The sequence of polynomials
$\displaystyle \{
\mathcal{B}_{m}^{(1)}\}_{m \in {\mathbb{N}}}$ given by
 \begin{equation*}
\mathcal{B}_{m}^{(1)}(z):=\mathcal{U}_{x}\left(
\frac{V_{m+1}(z)-V_{m+1}(h(x) )}{z-h(x)}\,\mathcal{
P}_{0}(x)\right),
 \end{equation*}
is said to be the \textit{sequence of associated polynomials of the first kind for} $\displaystyle \{
\mathcal{B}_{m} \}_{m \in {\mathbb{N}}}$ and  $\displaystyle \mathcal{U}$.
In a similar way, the sequence of polynomials
$\displaystyle \{G_{m}^{(1)}\}_{m \in {\mathbb{N}}}$ given by
\begin{equation*}
G^{(1)}_{m} (z)= \left[ \left( \frac{G_{m+1}^T
(z)- G_{m+1}^T (h(x))}{z-h(x)}\right){\mathcal{U}}_x
\right]({\mathcal{P}}_0(x)) ,
\end{equation*}
is said to be the \textit{sequence of associated polynomials of the first kind for} $\{G_{m}\}_{m \in {\mathbb{N}}}$ and  $\displaystyle \mathcal{U}$.
Here $\displaystyle \mathcal{U}_{x}$ represents the action of
$\displaystyle \mathcal{U}$ on the variable $\displaystyle x$.
\end{defi}
\begin{teo} \label{teohpesq}
Let ${\mathcal{U}}$ be a quasi-definite vector of linear functionals, $\{{\mathcal{B}}_{m}\}_{m \in {\mathbb{N}}}$ a vector sequence of polynomials, $\{{\mathcal{B}}^{(1)}_m\}_{m \in {\mathbb{N}}}$ its sequence of associated polynomials of the first kind,
and ${\mathcal{F}}$ the generalized Markov function given in ~\eqref{fmarkov}. Then $\{{\mathcal{B}}_{m}\}_{m \in {\mathbb{N}}}$
is left-orthogonal with respect to the vector of linear functionals ${\mathcal{U}}$ if and only if
$$V_{m+1}(z) {\mathcal{F}}(z) -
{\mathcal{B}}^{(1)}_m (z) = \Delta_{m+1} \frac{1}{z^{m+2}}+ \cdots
\, \, .$$
\end{teo}
\begin{proof} From the definition of
${\mathcal{B}}^{(1)}_m $, we get
\begin{equation*}
{\mathcal{B}}^{(1)}_m (z) =
{\mathcal{U}}_x \left( \frac{V_{m+1} (z)- V_{m+1} (h(x))
}{z-h(x)}{\mathcal{P}}_0(x) \right) =  V_{m+1} (z)
{\mathcal{F}}(z) -  {\mathcal{U}}_x \left(
\frac{{\mathcal{B}}_{m+1}(x)}{z-h(x)} \right) \, .
\end{equation*}
But
\begin{eqnarray*}
{\mathcal{U}}_x \left(
\frac{{\mathcal{B}}_{m+1}(x)}{z-h(x)}
\right) = {\mathcal{U}}_x\left( \sum_{n=0}^\infty
\frac{(h(x))^n}{z^{n+1}}
{\mathcal{B}}_{m+1}(x)\right) = \sum_{n=0}^\infty \frac{((h(x))^n
{{\mathcal{U}}_x})({\mathcal{B}}_{m+1}(x))}{z^{n+1}}.
\end{eqnarray*}
Hence,
\begin{eqnarray*}
{\mathcal{U}}_x
\left(\frac{{\mathcal{B}}_{m+1}(x)}{z-h(x)} \right) =
\sum_{n=m+1}^\infty \frac{((h(x))^n
{{\mathcal{U}}_x})({\mathcal{B}}_{m+1}(x))}{z^{n+1}} = \Delta_{m+1}
\frac{1}{z^{m+2}}+ \cdots,
\end{eqnarray*}
 if and only if the sequence
$\{{\mathcal{B}}_{m}\}_{m \in {\mathbb{N}}}$ is left-orthogonal with respect to ${{\mathcal{U}}_x}$.
\end{proof}
\begin{teo}
Let ${\mathcal{U}}$ be a quasi-definite vector of linear functionals,
$\{G_{m}\}_{m \in {\mathbb{N}}}$ a sequence of matrix polynomials with $N\times N$ matrix coefficients,  $\{G^{(1)}_m\}_{m \in {\mathbb{N}}}$ its sequence of associated polynomials of the first kind, and ${\mathcal{F}}$ is the generalized Markov function. Then $\{G_m\}_{m \in {\mathbb{N}}}$ is right-orthogonal with respect to the vector of linear functionals ${\mathcal{U}}$ if and only if
$${\mathcal{F}}(z) G_{m+1}(z) - G^{(1)}_{m}(z) = \Theta_{m+1} \frac{1}{z^{m+2}}+ \cdots \, \, .$$
\end{teo}
\begin{proof}
Taking into account the definition of the polynomial $\{G^{(1)}_{m}\}_{m \in {\mathbb{N}}}$, we have
 \begin{eqnarray*} G^{(1)}_{m} (z)
 &=& \left[ \left(
\frac{G_{m+1}^T (z)- G_{m+1}^T (h(x))}{z-h(x)}\right){\mathcal{U}}_x
\right]({\mathcal{P}}_0(x))\\
&=& {\mathcal{U}}_x \left(
\frac{{\mathcal{P}}_0(x)}{z-h(x)}\right)G_{m+1}(z) - \left(G_{m+1}^T
(h(x)){\mathcal{U}}_x\right)\left(
\frac{{\mathcal{P}}_0(x)}{z-h(x)}\right) \, .
\end{eqnarray*}
But,
 \begin{eqnarray*}
  \left(G_{m+1}^T (h(x)){\mathcal{U}}_x
\right)\left( \frac{{\mathcal{P}}_0(x)}{z-h(x)}\right)&=&
\sum_{n=0}^\infty
\frac{1}{z^{n+1}}({\mathcal{L}}_{m+1}^x) \left( {\mathcal{P}}_n(x)
\right).
 \end{eqnarray*}
Hence,
\begin{eqnarray*}
\left(G_{m+1}^T (h(x)){\mathcal{U}}_x \right)\left(
\frac{{\mathcal{P}}_0(x)}{z-h(x)}\right)=  \Theta_{m+1}
\frac{1}{z^{m+2}}+ \cdots,
\end{eqnarray*}
if and only if the sequence  $\{G_{m}\}_{m \in {\mathbb{N}}}$ is right-orthogonal with respect to ${{\mathcal{U}}_x}$.
\end{proof}
$\phantom{ola}$Next, some algebraic results concerning  the behavior of the sequences of matrix orthogonal polynomials $\{V_m\}_{m \in {\mathbb{N}}}$ and $\{G_m\}_{m \in {\mathbb{N}}}$ are given.
\begin{teo}
Let $h$ be a polynomial of fixed degree $N$ and ${\mathcal{U}}$ be a quasi-definite vector of linear functionals.
Let $\{G_{m}\}_{m \in {\mathbb{N}}}$ and $\{{\mathcal{B}}_m\}_{m \in {\mathbb{N}}}$ be, respectively, sequences of matrix polynomials with $\deg G_m =m$, for all $m \in {\mathbb{N}}$ and
${\mathcal{B}}_m(x)=V_m(h(x)){\mathcal{P}}_0(x)$, where $V_m$ is a matrix polynomial with $\deg V_m =m$, for all $m \in {\mathbb{N}}$. Then, the following statements are equivalent:
\begin{itemize}
\item[(a)] $\{\mathcal{B}_m\}_{m \in {\mathbb{N}}}$ is a sequence of vector polynomials left-orthogonal with respect to ${\mathcal{U}}$.
\item[(b)] $\{{\mathcal{L}}_n\}_{n \in {\mathbb{N}}}$ is a sequence of vector linear functionals bi-orthogonal with respect to $\{\mathcal{B}_m\}_{m \in {\mathbb{N}}}$ such that ${\mathcal{L}}_n=G_n^T(h(x)){\mathcal{U}}$.

\item[(c)] $\{V_m\}_{m \in {\mathbb{N}}}$ and $\{G_m\}_{m \in {\mathbb{N}}}$ satisfy the Christoffel-Darboux type formula
\begin{eqnarray} \label{cdm1}
(x-z)\sum_{k=0}^m G_k(z)V_k(x)=G_m(z)A_m
V_{m+1}(x)-G_{m+1}(z) C_{m+1}V_m(x),
\end{eqnarray}
with $x, \, z \in {\mathbb{C}}$.

\item[(d)] For every $m \in {\mathbb{N}}$, $\{V_m\}_{m \in {\mathbb{N}}}$ and $\{G_m\}_{m \in {\mathbb{N}}}$  satisfy  the confluent formula
\begin{gather} \label{conscdm1}
G_m(x)A_m
V_{m+1}(x)-G_{m+1}(x) C_{m+1}V_m(x)=0_{N \times N},
 \\ \label{cdm11}
\sum_{k=0}^m
G_k(x)V_k(x)=G_m(x)A_m V_{m+1}^\prime(x)-G_{m+1}(x)C_{m+1}
V_m^\prime(x),
\end{gather}

\item[(e)]  For every $m \in {\mathbb{N}}$, $\{V_m\}_{m \in {\mathbb{N}}}$ and $\{G_m\}_{m \in {\mathbb{N}}}$ satisfy for all
$x \in {\mathbb{C}}$
\begin{gather} \label{conscdm2}
G_m(x)A_m
V_{m+1}(x)-G_{m+1}(x) C_{m+1}V_m(x)=0_{N \times N},
 \\ \label{cdm12}
\sum_{k=0}^m G_k(x)V_k(x)=G_{m+1}^\prime(x)C_{m+1}
V_m(x)-G_m^\prime(x) A_m V_{m+1}(x) \, .
\end{gather}
\end{itemize}
\end{teo}
\begin{proof}
To prove this theorem we will proceed according to the following scheme
$(a)\Leftrightarrow(b)$, $(b)\Rightarrow(c)\Rightarrow(e)$, $(e)\Rightarrow(b)$ and $(c) \Rightarrow (d) \Rightarrow (a)$.

$\phantom{ola}$The equivalence $(a)\Leftrightarrow(b)$ is proved in Theorem~\ref{osvm}. To prove that (b) implies (c) we remember that the sequences of matrix polynomials  $\{V_m\}_{m \in {\mathbb{N}}}$ and $\{G_m\}_{m \in {\mathbb{N}}}$ verify,
respectively, the recurrence relations
\begin{eqnarray}
\label{r1} xV_m(x)&=&A_mV_{m+1}(x)+B_m V_m(x)+ C_m V_{m-1}(x)\\ 
\label{r2} z G_m(z)&=&G_{m-1}(z)A_{m-1}+G_m(z)B_m +G_{m+1}(z)
C_{m+1}
\end{eqnarray}
Multiplying on the left by $G_m(z)$ in both sides of~\eqref{r1}  and
on the right by $V_m(x)$ in both sides of ~\eqref{r2} and subtracting the resulting expressions, we get
\begin{multline*}
(x-z)G_m(z)V_m(x)=\left[G_m(z)A_mV_{m+1}(x)-G_{m-1}(z)A_{m-1}V_m(x)\right] \\ - \left[
G_{m+1}(z)C_{m+1}V_m(x)-G_m(z)C_mV_{m-1}(x)\right]
\end{multline*}
and so we have~\eqref{cdm1}.
To prove that (c) implies (d), we just have to take
$z=x$ in~\eqref{cdm1} and then we obtain~\eqref{conscdm1}. ~\eqref{cdm11} follows from~\eqref{cdm1} by differentiation with respect to $x$ and letting  $z=x$.

$\phantom{ola}$To prove that (c) implies (e), we must take $z=x$ in~\eqref{cdm1} and then~\eqref{conscdm2} holds.~\eqref{cdm12} follows in a similar way by differentiating~\eqref{cdm1} with respect to $z$ and taking
$z=x$.

$\phantom{ola}$To complete the proof we need to show that (d) implies
(a). We can rewrite the equation~\eqref{cdm11} as
$$ G_m(x)A_m V_{m+1}^\prime(x)-G_{m+1}(x)C_{m+1} V_m^\prime(x)\\
= G_m(x)V_m(x) + \sum_{k=0}^{m-1} G_k(x)V_k(x),$$ or,
equivalently,
\begin{multline*}
 G_m (x) V_m(x)= G_m(x)[ A_m V_{m+1}^\prime(x) +C_{m}
V_{m-1}^\prime(x) ] \\- [G_{m+1}(x)C_{m+1}+ G_{m-1}(x)A_{m-1}]
V_{m}^\prime(x).
\end{multline*}
Using~\eqref{conscdm1} we get
$\displaystyle [(A_mV_{m+1}(x)+ C_m V_{m-1}(x))V_m^{-1}(x)]^\prime= I_{N \times N} \, $.
Then, we have  $$[A_m V_{m+1}(x)+ C_m V_{m-1}(x)]V_m^{-1}(x)= x
I- B_n,$$
i.e, $\{V_m\}_{m \in {\mathbb{N}}}$ satisfies a three-term recurrence relation.
Now, multiplying both sides in the three-term recurrence relation by ${\mathcal{P}}_0$, from the definition of ${\mathcal{B}}_m$ and by Theorem~\ref{FAVARD1}, the result follows.

$\phantom{ola}$Finally, to prove that $(e)\Rightarrow (b)$ we proceed in a similar way as in the proof of $(d) \Rightarrow (a)$ starting from~\eqref{cdm12} and taking into account Theorem~\ref{osvm}.
\end{proof}

\section{Markov type theorem}

$\phantom{ola}$The block matrix
\begin{eqnarray*} 
 J=\left[%
\begin{array}{cccc}
  B_0 & A_0 & 0_{N \times N} &  \\
  C_1 & B_1 & A_1 & \ddots \\
  0_{N \times N} & C_2 & B_2 & \ddots \\
    & & \ddots& \ddots \\
\end{array}%
\right],
\end{eqnarray*}
is related to the matrix polynomial
sequences $\{V_m\}_{m \in {\mathbb{N}}}$ and $\{G_m\}_{m \in {\mathbb{N}}}$ trough the recurrence relations
\eqref{rrvm} and~\eqref{rrgm}. This block matrix is said to be the $N$-block
Jacobi matrix associated with the above matrix polynomial sequences.

$\phantom{ola}$For polynomials  satisfying a symmetric recurrence relation, it was proved in~\cite{DurLR} that the zeros of the $m$-th matrix orthogonal polynomial are the eigenvalues of the leading principal submatrix $J_{mN}$ of $J$. This result can
be generalized for sequences of orthogonal polynomials that satisfy
non-symmetric recurrence relations. Thus, for $m \in {\mathbb{N}}$,
the zeros of the matrix polynomials $G_{m}$ and $V_{m}$ are the eigenvalues of the matrix $J_{mN}$ (with the same order of algebraic
multiplicity) where $I_{mN \times mN}$ is the~$mN\times mN$ identity matrix and $J_{mN}$ is the leading principal submatrix of
dimension~$mN\times mN$ for the $N$-block Jacobi matrix.
\begin{lemma} \label{enri} \cite{Dur96}
Let $V(t)$ be a~$N\times N$ matrix polynomial and let
$\textbf{a}$ be a zero of $V(t)$ of multiplicity $p$, i.e, a zero of multiplicity $p$
of the scalar polynomial $\det V(t)$. Let
$$L(\textbf{a},V)=\{v \in {\mathbb{C}}^N: v V(\textbf{a})=0_{1 \times N}\} \ \mbox{and} \ R(\textbf{a},A)=\{v \in {\mathbb{C}}^N: V(\textbf{a})v^*=0_{N \times 1}\}.$$
If $\operatorname{dim} \, L(\textbf{a},V) = \operatorname{dim} \, R(\textbf{a},V)=p$, then
$$\left(\adj \,(V(t))\right)^{(l)}(\textbf{a})=0_{N \times N} \mbox{, for }
l=0,\ldots,p-2  \mbox{ and, }  \left(\adj\,
(V(t))\right)^{(p-1)}(\textbf{a})\neq 0_{N \times N}.$$
Moreover,
$\operatorname{rank} \left(\adj\, (V(t))\right)^{(p-1)}(\textbf{a})=p$ and $\left(\adj\, (V(t))\right)^{(p-1)}(\textbf{a})$ defines
a linear mapping from ${\mathbb{C}}^N$ onto $L(\textbf{a},V)$ which is
an isomorphism from $R(\textbf{a},V)$ into $L(\textbf{a}, V)$.
\end{lemma}
\begin{lemma}\label{frpropalternativa}
Let $x_{m,k}$, $k=1,\ldots, s$ with $s\leq mN$ be the zeros of the
matrix polynomial $V_m$.
For any matrix polynomial $V(t)$ of degree $\leq n-1$ we have the partial fraction decomposition
$t \in {\mathbb{C}}\setminus \{x_{m,1},\ldots,x_{m,s}\},$
$$V(t)(V_m(t))^{-1}=\sum_{k=1}^s \frac{C_{m,k}}{x-x_{m,k}}$$
where $\displaystyle C_{m,k}=\frac{l_k}{(\det\,
(V_m(t)))^{(l_k)}(x_{m,k})} V(x_{m,k})(\adj\,(V_m(t)))^{(l_k
-1)}(x_{m,k})$ and $l_k$ is the multiplicity of $x_{m,k}$, ($l_k \leq N$).
\end{lemma}
$\phantom{ola}$With these results we are able to establish a quadrature formula for
the matrix orthogonal sequence $\{V_m\}_{m \in {\mathbb{N}}}$.
\begin{teo}[Quadrature Formula]\label{formulaquadratura} Let $\{V_m\}_{m \in {\mathbb{N}}}$ be the sequence of matrix polynomials that is left-orthogonal with
respect to the matrix of measures $W$. Also let $\{\mathcal{B}_m\}_{m \in {\mathbb{N}}}$ be
the sequence of vector polynomials defined by~\eqref{vv} and let
$\{{\mathcal{B}}^{(1)}_{m}\}_{m \in {\mathbb{N}}}$, be the sequence of associated  polynomials of the first kind  for $\displaystyle \{
\mathcal{B}_{m} \}_{m \in {\mathbb{N}}} $ and  $\displaystyle \mathcal{U}$. Let $x_{m,k},$ $(k=1,\ldots,s)$ be
the zeros of the matrix polynomial $V_m$
(hence $s\leq mN$), and let $\Gamma_{m,k}$ be the matrices
\begin{eqnarray*}
\Gamma_{m,k}= \frac{l_k}{(\det\, (V_m(x)))^{(l_k)}(x_{m,k})}
(\adj\,(V_m(x)))^{(l_k -1)}(x_{m,k})
{\mathcal{B}}^{(1)}_{m-1}(x_{m,k}) \, ,
\end{eqnarray*}
for $k=1,\ldots,s$
where $l_k$ is the multiplicity $x_{m,k}$.

$\phantom{ola}$Then, for any polynomial $V$ of degree less than  or equal to $2m-1$
the following quadrature formula holds
\begin{eqnarray*}
 \int
V(h(x))dW(h(x)) = \sum_{k=1}^s V(x_{m,k})\Gamma_{m,k} \, .
\end{eqnarray*}
\end{teo}
\begin{proof}
Let $V$ be a matrix polynomial of  degree less than  or equal to $2m-1$.
Since $V_m$ is a polynomial with non-singular leading coefficient, then (cf.~\cite{gant})
$$V(x)=C(x)V_m(x)+R(x),$$ where $C$ and $R$
are matrix polynomials with degree of $R$ less than  or equal to $m-1$.
Thus
$\displaystyle V(x)V^{-1}_m(x)=C(x)+R(x)V^{-1}_{m}(x) $
assuming that  $x$ is not a zero of $V_m$. Since degree $R(x)\leq m-1$, using Lemma~\ref{frpropalternativa}
we get $$R(x)V_m^{-1}(x)=\sum_{k=1}^s \frac{C_{m,k}}{x-x_{m,k}},$$
where the matrices $C_{m,k}$ are given by
$$C_{m,k}=\frac{l_k}{(\det\,(V_m(x)))^{(l_k)}(x_{m,k})}R(x_{m,k})(\adj\,(V_m(x)))^{(l_k-1)}(x_{m,k}).$$
According to Lemma~\ref{enri}, $V_m(x_{m,k})\left(\adj\,
(V_m(x))\right)^{(l_k-1)}(x_{m,k})=0_{N \times N}$ and taking into account that $R(x_{m,k})= V(x_{m,k})-C(x_{m,k})V_m(x_{m,k})$, the previous expression becomes
$$C_{m,k}=\frac{l_k}{(\det\,(V_m(x)))^{(l_k)}(x_{m,k})}V(x_{m,k})(\adj\,(V_m(x)))^{(l_k-1)}(x_{m,k}).$$
Then, $$V(x)=C(x)V_m(x) + \sum_{k=1}^s C_{m,k}
\frac{V_m(x)}{x-x_{m,k}}.$$ Since
$$V_m(x_{m,k})\left(\adj
(V_m(x))\right)^{(l_k-1)}(x_{m,k})=\left(\adj
(V_m(x))\right)^{(l_k-1)}(x_{m,k})V_m(x_{m,k})=0_{N \times N}$$
we have
 $$V(x)=C(x)V_m(x) + \sum_{k=1}^s C_{m,k}
\frac{V_m(x)-V_m(x_{m,k})}{x-x_{m,k}}.$$ Taking $x=h(t)$, we have
$$V(h(t))=C(h(t))V_m(h(t)) + \sum_{k=1}^s C_{m,k}
\frac{V_m(x_{m,k})-V_m(h(t))}{x_{m,k}-h(t)}.$$
Then, from the integral representation of the associated polynomials of the first kind
$${\mathcal{B}}_{m-1}^{(1)}(z)=\int \frac{V_m(z)-V_m(h(x))}{z-h(x)}dW(h(x)),$$ it follows that
$$\int V(h(t))dW(h(t))=\int C(h(t))V_m(h(t))dW(h(t)) + \sum_{k=1}^s C_{m,k}
{\mathcal{B}}_{m-1}^{(1)}(x_{m,k}).$$
So, from the orthogonality of $\{V_m\}_{m \in {\mathbb{N}}}$ with respect to $W$ we have
$$\int V(h(t))dW(h(t))= \sum_{k=1}^s C_{m,k}
{\mathcal{B}}_{m-1}^{(1)}(x_{m,k}),$$ and the statement follows.
\end{proof}
$\phantom{ola}$The next result is an extension of one proved by A. J. Dur\'{a}n in~\cite{Dur96}.
It deals with the ratio asymptoticss of the  $m$-th orthogonal polynomial $V_m$
with respect to the generalized Markov matrix function, ${{\mathcal{F}}}$, and the $(m-1)$-th
associated polynomial of the first kind ${\mathcal{B}}^{(1)}_{m-1}$.
\begin{teo}[Generalized Markov's theorem]
Let $\mathcal{U}$ be a quasi-definite vector of linear functionals, $\{V_m\}_{m \in {\mathbb{N}}}$ be the sequence of matrix polynomials left-orthogonal with respect to the generalized Markov matrix function, ${{\mathcal{F}}}$, defined by~\eqref{fmarkov}, and let $\{{\mathcal{B}}^{(1)}_m\}_{m \in {\mathbb{N}}}$ be the sequence of associated polynomials of the first kind for
$\displaystyle \{\mathcal{B}_{m}\}_{m \in {\mathbb{N}}}$ and $\displaystyle \mathcal{U}$.
Then,
$$\lim_{m\rightarrow \infty} V^{-1}_m(z) {\mathcal{B}}^{(1)}_{m-1} (z) = {\mathcal{F}}(z)$$
locally uniformly in ${\mathbb{C}}\setminus \Gamma$, where
$\displaystyle \Gamma=\cap_{N\geq 0} M_N,\quad M_N=\overline{\cup_{n\geq N}\{\mbox{zeros of }\,
V_m \}} \, .$
\end{teo}
\begin{proof}
First, from Lemma~\ref{frpropalternativa} we get
$$V^{-1}_m(z) {\mathcal{B}}^{(1)}_{m-1} (z)=\sum_{k=1}^s \Gamma_{m,k}
\frac{1}{z-x_{m,k}},$$ where $\Gamma_{m,k}$ are the matrix coefficients that appear in the quadrature formula presented in Theorem~\ref{formulaquadratura} and  $x_{m,k}$ are the zeros of  $V_m$.
On the other hand, there always exist complex numbers $y_{m,k}$ such that
$h(y_{m,k})=x_{m,k},$ and
$$V^{-1}_m(z) {\mathcal{B}}^{(1)}_{m-1} (z)=\sum_{k=1}^s \Gamma_{m,k}
\frac{1}{z-h(y_{m,k})}.$$
We consider the sequence of discrete matrices of measures $\{\mu_m\}_{m \in \mathbb{N}}$ defined by
$$\mu_m=\sum_{k=1}^{s} \Gamma_{m,k} \, \delta_{y_{m,k}} \, . $$
Thus,
\begin{eqnarray} \label{amil}
V^{-1}_m(z) {\mathcal{B}}^{(1)}_{m-1} (z)=\sum_{k=1}^s \Gamma_{m,k}
\frac{1}{z-h(y_{m,k})}=\int \frac{d\mu_m(h(x))}{z-h(x)}
\end{eqnarray}
if $z$ is not a zero of $V_m$.
Taking into account~\eqref{amil}, it will be enough to prove that
$$\lim_{m \rightarrow\infty} \int \frac{d\mu_m(h(x))}{z-h(x)}= {\mathcal{F}}(z)
\mbox{ for } \ z \in {\mathbb{C}} \setminus \Gamma \, .$$

$\phantom{ola}$The first step deals with the pointwise convergence. Otherwise, we assume that there exists a complex number $z\in \, {\mathbb{C}}\setminus \Gamma$, an increasing sequence of nonnegative integers $\{m_{l}\}_{l \in {\mathbb{N}}},$ and a positive constant $C$ such that
\begin{eqnarray} \label{eq:nova}
\left\| \int \frac{d\mu_{m_l}(h(x))}{z-h(x)} - {\mathcal{F}}(z)
\right\|_2 \geq C >0, \,\, l\geq0,
\end{eqnarray}
where $\|\, . \,\|_2$ denotes the spectral norm of a matrix, i.e.,
$$\|A\|_2=\mbox{max}\{\sqrt \lambda: \, \lambda \, \mbox{is a eigenvalue of}\, A^* A\}.$$
Taking an increasing sequence $\{a_k\}_{k \in {\mathbb{N}}}$ such that $a_k \rightarrow \infty$, and
using the Banach-Alaoglu's theorem there exists a subsequence $\{r_l\}_{l \in {\mathbb{N}}}$ from $\{m_l\}_{l \in {\mathbb{N}}}$,
defined on a curve $\gamma_k $ contained in a disc $|z| < a_k$, with the
same $k$-th moment of the vector of linear functionals,~$\mathcal{U}$, for $k\leq 2r_l-1$,
such that
\begin{eqnarray}
 \label{amil2}
 \lim_{l\rightarrow \infty} \int_{\gamma_k} f(h(x))d\mu_{r_l}(h(x))=  \frac{1}{2 i \pi}
 \int_{\gamma_k} f(h(z)) \mathcal{U}_x \left( \frac{\mathcal{P}_0(x)}{z - h (x)} \right) \, dz \, .
 \end{eqnarray}
Moreover,
\begin{equation*}
\left\| \int \frac{d\mu_{r_l}(h(x))}{z-h(x)} - {\mathcal{F}}(z) \right\|_2
 \leq
\left\| \int_{\gamma_k} \frac{d\mu_{r_l}(h(x))}{z-h(x)}- {\mathcal{F}}(z) \right\|_2
 \\ + \left\| \int_{\ell_k} \frac{d\mu_{r_l}(h(x))}{z-h(x)} \right\|_2 \, ,
\end{equation*}
with $\ell_k $ in the exterior of the disc $|z| < a_k$.
We write $S_0$ for the first moment of the matrices of measures $\mu_{r_l}$ which is the first moment of $\mathcal{U}$.
Then, by taking $k$ and then $r_l$ large enough, from~\eqref{eq:nova} and~\eqref{amil2} we obtain
\begin{eqnarray*}
\frac{C}{2} & \leq & \max \left(\frac{1}{|z-h(a_k)|}\right) \left\|
\int_{\ell_k} d \mu_{r_l}(h(x)) \right\|_2
 \\
& \leq & \max \left(\frac{1}{|z-h(a_k)|}\right) \| S_0 \|_2 \, .
\end{eqnarray*}
But this yields $C=0$ and, therefore,~\eqref{eq:nova} is not possible.

$\phantom{ola}$The next step is to prove that the analytic functions which are the entries of the matrix
$\displaystyle\int \frac{d\mu_{m}(h(x))}{z-h(x)}$ are uniformly bounded in compact sets of
${\mathbb{C}} \setminus \Gamma$. Then, the uniform convergence in compact subsets of ${\mathbb{C}} \setminus \Gamma$
will follow from Stieltjes-Vitali's theorem.

$\phantom{ola}$Given a compact $K \subset\,{\mathbb{C}} \setminus \Gamma,$ let notice that $K \cap M_N \neq \emptyset$,
for $N$ large enough, and then there exists $A > 0$ such that
$$\left|\frac{1}{z-h(x)}\right|\leq A ,\,\, \mbox{for} \,\, z \in K \,\,\mbox{and}\,\, h(x) \in M_N.$$
Then, for $n \geq N$ 
$$\left\| \int \frac{d\mu_n (h(x))}{z-h(x)} \right\| 
\leq A \, S_0 \, .$$

$\phantom{ola}$The spectral norm $\displaystyle \int \frac{d\mu_m(h(x))}{z-h(x)}$ is uniformly
bounded and, therefore, from the equivalence of the norms in finite dimensional spaces, the result follows.
\end{proof}
\begin{rem}
In an analog way we can deduce the following result.
Let $\{G_m\}_{m \in {\mathbb{N}}}$ be the sequence of matrix polynomials right-orthogonal with respect to the generalized Markov function ${\mathcal{F}}$ and let $\{G^{(1)}_m\}_{m \in {\mathbb{N}}}$ be the sequence of associated polynomials of the first kind for $\{G_m\}_{m \in {\mathbb{N}}}$ and ${\mathcal{U}}$. Then,
$$\lim_{m\rightarrow \infty} G^{(1)}_{m - 1} (z) G^{- 1}_{m}(z)= {\mathcal{F}}(z),$$
for $z \in {\mathbb{C}} \setminus \Gamma$ and the convergence is locally uniformly on  ${\mathbb{C}} \setminus \Gamma$, where
$$\Gamma = \cap_{N\geq 0} M_N , \quad M_N = \overline{\cup_{n\geq N}\{\mbox{zeros of }
G_m \}}.$$
\end{rem}

\ifx\undefined\bysame
\newcommand{\bysame}{\leavevmode\hbox to3em{\hrulefill}\,}
\fi

\section*{Acknowledgements}
The work of the second author (FM) has been  supported by Direcci\'on General de Investigaci\'on, Ministerio de Ciencia e Innovaci\'on of Spain, under grant MTM2009-12740-C03-01.

\end{document}